\documentclass[12pt,leqno]{amsart}

\usepackage[bookmarksopen,bookmarksdepth=3,colorlinks,citecolor=red,pagebackref,hypertexnames=true]{hyperref}
\usepackage[backrefs,msc-links,nobysame,initials]{amsrefs}

\usepackage[textsize=tiny]{todonotes}

\usepackage{xfrac}

\usepackage[T1]{fontenc}
\usepackage{lmodern}

\rmfamily
\usepackage{graphicx}
\usepackage{transparent}

\usepackage{mathtools,MnSymbol}
\mathtoolsset{showonlyrefs,showmanualtags}

\allowdisplaybreaks

\theoremstyle{plain}

\newtheorem{lemma}[equation]{Lemma}
\newtheorem{proposition}[equation]{Proposition}
\newtheorem{theorem}[equation]{Theorem}
\newtheorem*{theorem*}{Theorem}
\newtheorem{corollary}[equation]{Corollary}
\theoremstyle{definition}
\newtheorem{definition}[equation]{Definition}

\newtheorem{remark}[equation]{Remark}

\numberwithin{equation}{section}

\def\norm#1.#2.{\lVert#1\rVert_{#2}}
\def\Norm#1.#2.{\bigl\lVert#1\bigr\rVert_{#2}}
\def\NOrm#1.#2.{\Bigl\lVert#1\Bigr\rVert_{#2}}
\def\NORm#1.#2.{\biggl\lVert#1\biggr\rVert_{#2}}
\def\NORM#1.#2.{\Biggl\lVert#1\Biggr\rVert_{#2}}

\def\diam{\textnormal{diam}}
\def\ind{\textnormal{\textbf 1}}

\def\R{\mathbb R}

\def\dil{\textnormal{dil}}
\DeclareMathOperator*{\esssup}{ess\,sup}
\def\dist{\operatorname{dist}}
\def\diam{\operatorname{diam}}

\def\ip#1,#2,{\langle #1,#2\rangle}
\def\Ip#1,#2,{\bigl\langle#1,#2\bigr\rangle}
\def\IP#1,#2,{\Bigl\langle#1,#2\Bigr\rangle}

\def\Abs#1{\bigl\lvert#1\bigr\rvert}

\begin{document}
\raggedbottom
%%%%%%%%%%%%%%%%%%%%%%%%%%%%%  Title
\title[Tauberian conditions and Muckenhoupt weights]{Tauberian conditions, Muckenhoupt weights, and differentiation properties of weighted bases}

\author[P. Hagelstein]{Paul Hagelstein} \address{Paul Hagelstein, Department of Mathematics, Baylor University, Waco, Texas 76798, USA.}
\email{paul\_hagelstein@baylor.edu}

\author[T. Luque]{Teresa Luque}\address{Teresa Luque, Departamento de Analisis Matem\'atico, Facultad de Matem\'aticas, Universidad de Sevilla, 41080 Sevilla, Spain}
\email{tluquem@us.es}

\author[I. Parissis]{Ioannis Parissis}
\address{Ioannis Parissis.: Department of Mathematics, Aalto University, P.O.Box 11100, FI-00076 Aalto, Finland} \email{ioannis.parissis@gmail.com}
%\urladdr{\href{http://www.helsinki.fi/~parissis}{http://www.helsinki.fi/~parissis}}

\thanks{P.H. is partially supported by the Simons Foundation grant 208831.}
\thanks{T.L. is supported by the Spanish Ministry of Economy and Competitiveness grant BES-2010-030264}
\thanks{I.P. is supported by the Academy of Finland, grant 138738.}

\begin{abstract}  Let $\mathfrak{B}$ be a homothecy invariant collection of convex sets in $\mathbb{R}^{n}$.   Given a measure $\mu$, the associated weighted geometric maximal operator $M_{\mathfrak{B}, \mu}$ is defined by
\begin{align*}
      M_{\mathfrak{B}, \mu}f(x) \coloneqq \sup_{x \in B \in \mathfrak{B}}\frac{1}{\mu(B)}\int_{B}|f|d\mu.
\end{align*}
      It is shown that, provided $\mu$ satisfies an appropriate doubling condition with respect to $\mathfrak{B}$ and $\nu$ is an arbitrary locally finite measure, the maximal operator $M_{\mathfrak{B}, \mu}$ is bounded on $L^{p}(\nu)$ for sufficiently large $p$ if and only if it satisfies a Tauberian condition of the form
\begin{align*}
      \nu\big(\big\{x \in \mathbb{R}^{n} : M_{\mathfrak{B}, \mu}(\ind_E)(x) > \frac{1}{2} \big\}\big) \leq c_{\mu, \nu}\nu(E).
\end{align*}
As a consequence of this result we provide an alternative characterization of the class of Muckenhoupt weights $A_{\infty, \mathfrak{B}}$ for homothecy invariant Muckenhoupt bases $\mathfrak{B}$ consisting of convex sets.  Moreover, it is immediately seen that the strong maximal function $M_{\mathfrak{R}, \mu}$, defined with respect to a product-doubling measure $\mu$, is bounded on $L^{p}(\nu)$ for some $p > 1$ if and only if
\begin{align*}
\nu\big(\big\{x \in \mathbb{R}^{n} : M_{\mathfrak{R}, \mu}(\ind_E)(x) > \frac{1}{2}\big\}\big) \leq c_{\mu, \nu}\nu(E)\;
\end{align*}
holds for all $\nu$-measurable sets $E$ in $\mathbb{R}^{n}$.
 In addition, we discuss applications in differentiation theory, in particular proving that a $\mu$-weighted homothecy invariant basis of convex sets satisfying appropriate doubling and Tauberian conditions  must differentiate $L^{\infty}(\nu)$.

\end{abstract}
\maketitle
%%%%%%%%%%%%%%%%%%%%%%%%%%%%%  Title

%%%%%%%%%%%%%%%%%%%%%%%%%%%%%% SECTION  SECTION SECTION
\section{Introduction} \label{s.introduction}
The study of weighted inequalities for classical operators such as the Hardy-Littlewood maximal function and the Hilbert transform has been an active area of research in harmonic analysis since the seminal paper of Muckenhoupt, \cite{Muck}. Here, by weighted inequalities we mean the study of the boundedness properties of an operator $T$ on some weighted Lebesgue space $L^p(w)$, where $w$ is an appropriate non-negative, locally integrable function, that is, a \emph{weight}. Indeed, the weights $w$ for which the Hardy-Littlewood maximal function, the Hilbert transform, as well as more general Calder\'on-Zygmund operators are bounded on $L^p(w)$ were identified in \cite{Muck} as well as in the subsequent works \cites{HMW,CoFe}; these investigations led to the definition of the $A_p$ classes of weights; see Definition~\ref{d.Ap}. The first quantitative estimate of the operator norm $\|M\|_{L^p(w)}$ on the \emph{$A_p$-constant of the weight} was proved in \cite{Buckley} for the Hardy-Littlewood maximal function $M$. In recent years, the corresponding question concerning the sharp dependence of the norm of a Calder\'on-Zygmund operator $\|T\|_{L^p(w)}$ on the $A_p$-constant of the weight has spurred an overwhelming amount of activity and development of relevant tools. Important representatives of these results include (but are not exhausted to) the work of Petermichl in \cite{Peter}, where the sharp bound for the Hilbert transform is proved, as well as the resolution of the celebrated $A_2$ conjecture by Hyt\"onen in \cite{Hyt} where the sharp bound is exhibited for general Calder\'on-Zygmund operators. Subsequent important developments and simplifications of the proof of the $A_2$ theorem can be found in \cite{HLP} and \cite{lernera2}. The two-weight problem for the Hilbert transform was also a notoriously hard problem, asking for necessary and sufficient conditions on a pair of weights $(v,w)$ so that the Hilbert transform is bounded from $L^p(w)$ to $L^p(v)$. The two-weight inequalities have been intensively investigated in several papers which led to the very recent resolution by Lacey in \cite{lacey}, following previous results by Lacey, Sawyer, Shen and Uriarte-Tuero in \cite{LSSU}.

All the results mentioned so far concern the classical or one-parameter theory, where the operators under study commute with one-parameter dilations of $\R^n$. On the other hand, the most basic example of the multi-parameter theory is the strong maximal function $M_\mathfrak R$, namely the maximal average of a  function with respect to $n$-dimensional rectangles with sides parallel to the coordinate axes. This operator is in many senses the prototype for multi-parameter analysis while there are also natural multi-parameter versions of the Hilbert transform as well as of more general singular integral operators. See for example \cite{RS}. As the terminology suggests, these operators commute with $n$-parameter dilations of $\R^n$. A general introduction to multi-parameter harmonic analysis is contained in \cite{Fef}. The basic weighted theory for the strong maximal function is also pretty well developed in a series of important papers such as \cite{F}, \cite{PiFe}, \cite{J}, \cite{JT} and \cite{Saw1}. A natural starting point for a more quantitative multi-parameter weighted theory would be the analogue of Buckley's theorem for the strong maximal function, namely, a sharp estimate on $\|M_{\mathfrak R}\|_{L^p(w)}$ in terms of the (strong) $A_p$-constant of the weight. No such quantitative estimate is currently known, a serious obstruction in carrying over the already described achievements of classical weighted theory to the multi-parameter setting.

A possible reason why such a sharp weighted theorem is elusive in the multi-parameter world is the failure of the Besicovitch covering argument for rectangles with arbitrary eccentricities. Indeed, it is an essential fact, underlying many of the sharp quantitative estimates in the classical weighted theory, that the (centered) Hardy-Littlewood maximal operator, defined with respect to a general measure, is \emph{always} bounded independently of the measure. This fails quite dramatically for the strong maximal function and this is just another manifestation of the failure of the Besicovitch covering argument. See \cite{F}. The relevance of this fact to proving bounds on the operator norms $\|M_{\mathfrak R} \|_{L^p(w)}$ is revealed by abstract theorems characterizing two-weight norm inequalities in terms of the boundedness of corresponding weighted maximal operators. For example it is implicit in Sawyer's two-weight norm inequalities for the Hardy-Littlewood maximal function in \cite{Saw} and for the strong maximal function in \cite{Saw1}. See also  \cite{J} and  \cite{lerner}.

%%%%%%%%%%%%%%%%%%%%%%%%%%%%%% SECTION  SECTION SECTION
\subsection*{The strong maximal function with respect to a measure} For the weighted strong maximal function $M_{\mathfrak R,w}$ one thus requires some restriction on the weight $w$ so that $M_{\mathfrak R,w}$ is bounded on $L^p(w)$. A sufficient condition is provided by the result of Fefferman, \cite{F}, that states that if a weight $w\in A_{\infty, \mathfrak R}$, that is if $w$ is a strong Muckenhoupt weight, then $M_{\mathfrak R, w}$ is bounded on $L^p(w)$. See \S \ref{s.overview} for precise definitions. It is thus no big surprise that many results in multi-parameter weighted theory begin with the assumption that $w\in A_{\infty,\mathfrak R}$. This assumption has also been used in order to derive the Fefferman-Stein inequality for the strong maximal function in \cite{Lin}, \cite{LP}, \cite{Mitsis} and \cite{Per1}. An apparently weaker substitute for the hypothesis $w\in A_{\infty,\mathfrak R}$, usually referred to as condition $(A)$, is that $w$ satisfies a \emph{Tauberian condition} of the form
\begin{align}\label{e.A}
\tag{$A$}	w(\{x\in\R^n: M_{\mathfrak R} (\ind_E)(x)>\frac{1}{2}\})\leq c w(E),
\end{align}
where $E\subset \R^n$ is any measurable set. This condition was introduced in \cite{J} in the context of weighted inequalities for quite general maximal functions. A consequence of our main theorem is however that \eqref{e.A} is equivalent to $w\in A_{\infty,\mathfrak R}$.

The previous discussion also motivates the seeking of conditions on a measure $\mu$ such that the strong maximal function $M_{\mathfrak R,\mu}$, defined with respect to the measure $\mu$, is bounded on $L^p(\mu)$. With more general results and precise definitions to follow, one of our main theorems implies:
%%%%%%%%%%%%%%%%%%%%%%%%%%%%%% THEOREM THEOREM THEOREM
\begin{theorem}\label{t.main} Let $\mu$ be a Borel non-negative measure which is locally finite and doubling with respect to rectangles with sides parallel to the coordinate axes. Then the operator $M_{\mathfrak R,\mu}$ satisfies
	\begin{align*}
		\mu(\{x\in\R^n: M_{\mathfrak R,\mu}(\ind_E)(x)>\frac{1}{2} \})\leq c \mu(E)
	\end{align*}
for every measurable set $E$ if and only if $M_{\mathfrak R,\mu}$ is bounded on $L^p(\mu)$ for some $p>1$.
\end{theorem}
%%%%%%%%%%%%%%%%%%%%%%%%%%%%%% THEOREM THEOREM THEOREM
This theorem can be thought of as a testing condition on the operator $M_{\mathfrak R,\mu}$. As it will become apparent, the constant $\frac{1}{2}$ is not important as it can be replaced by any fixed level $\gamma\in(0,1)$ in the statement of the theorem.

%%%%%%%%%%%%%%%%%%%%%%%%%%%%%% SECTION  SECTION SECTION
\subsection*{Differentiation with respect to bases of convex sets} A dual point of view and motivation for the investigations in this paper can be given in the language of differentiation theory. Given a collection of convex sets in $\R^n$ which is invariant under dilations and translations we want to study when the corresponding maximal operator differentiates $L^\infty(\R^n)$. It turns out that boundedness properties of quite general maximal operators can also be characterized in terms of Tauberian conditions in the spirit of \eqref{e.A}. Indeed, it is a classical result of Busemann and Feller, \cite{BF}, that a homothecy invariant basis $\mathfrak B$ consisting of open sets differentiates $L^\infty(\R^n)$ if and only if the corresponding maximal operator $M_{\mathfrak B}$ satisfies a Tauberian condition
\begin{align*}
 |\{x\in\R^n: M_{\mathfrak B}(\ind_E)(x)>\gamma\}|\leq c_\gamma |E|,
\end{align*}
for \emph{every} $\gamma \in(0,1)$ and for every measurable set $E$. See Theorem~\ref{t.bf} for the details. This point of view is discussed in detail in \cite{Guzdif} and taken up in \cite{HS}. In the last work it is shown that a homothecy invariant basis consisting of convex sets differentiates $L^\infty(\R^n)$ (with respect to the Lebesgue measure) if and only if it differentiates $L^p(\R^n)$ for some sufficiently large $p>1$. Note here that, lacking the convexity hypothesis on the basis one needs Tauberian conditions at all levels $\gamma \in (0,1)$. This should be contrasted to the results in \cite{HS} as well as in the current paper where the convexity assumption allows us to only assume a Tauberian condition at a \emph{fixed level}, say $\gamma=\frac{1}{2}$.

It is a natural question whether such results persist under the presence of a weight, or somewhat more generally, a measure $\mu$. More precisely, one seeks conditions on the measure $\mu$ and the basis $\mathfrak B$ so that the $\mu$-averages of a function $f\in L^p(\mu)$ converge to $f$ $\mu$-almost everywhere. If $\mathfrak B$ is some abstract basis of convex sets it is in general hard to tell whether $\mathfrak B$ differentiates functions in $L^p(\mu)$ with respect to $\mu$, for some $p>1$. For the basis of rectangles with sides parallel to the coordinate axes this is just a rephrasing of the question: ``When is the strong maximal function $M_{\mathfrak R,\mu}$, defined with respect to a measure $\mu$, bounded on $L^p(\mu)$ for some $p>1$?'' However our results address the more general case of homothecy invariant bases consisting of convex sets. Our most general theorem has the following form:

%%%%%%%%%%%%%%%%%%%%%%%%%%%%%% THEOREM THEOREM THEOREM
\begin{theorem} Let $\mathfrak B$ be a homothecy invariant basis consisting of open convex sets and assume that $\mu,\nu$ are locally finite, non-negative Borel measures on $\R^n$. Assume further that the measure $\mu$ is doubling with respect to $\mathfrak B$. Let $M_{\mathfrak B,\mu}$ denote the maximal operator with respect to the basis $\mathfrak B$ and the measure $\mu$. Then $M_{\mathfrak B,\mu}$ satisfies
	\begin{align*}
		\nu(\{x\in\R^n: M_{\mathfrak B,\mu}(\ind_E)(x)>\frac{1}{2} \})\leq c \nu(E)
	\end{align*}
if and only if $M_{\mathfrak B,\mu}$ is bounded on $L^p(\nu)$ for some $p>1$.
\end{theorem}
%%%%%%%%%%%%%%%%%%%%%%%%%%%%%% THEOREM THEOREM THEOREM
Among other consequences, we get as a corollary a ``weighted'' version of the Busemann-Feller theorem:
%%%%%%%%%%%%%%%%%%%%%%%%%%%%%% COROLLARY COROLLARY COROLLARY
\begin{corollary} Let $\mathfrak B$ be a homothecy invariant basis consisting of convex sets and $\mu,\nu$ be locally finite, non-negative measures on $\R^n$. Assume in addition that $\mu$ is doubling with respect to $\mathfrak B$. If
\begin{align*}
		\nu(\{x\in\R^n: M_{\mathfrak B,\mu}(\ind_E)(x)>\frac{1}{2} \})\leq c \nu(E)
\end{align*}
then $\mathfrak B$ differentiates $L^\infty(\nu)$ with respect to the measure $\mu$.
\end{corollary}
%%%%%%%%%%%%%%%%%%%%%%%%%%%%%% COROLLARY COROLLARY COROLLARY

%%%%%%%%%%%%%%%%%%%%%%%%%%%%%% SECTION  SECTION SECTION
\section{Notations} A few words concerning the notation used in this paper are necessary. Due to the technical nature of some of the proofs the notation becomes quite cumbersome but we have tried to be consistent with our choice of symbols. The current paragraph can be used as a guide to the notation and the reader is encouraged to consult it in order to avoid any sort of confusion.

We write $A\lesssim B$ whenever $A\leq C B$ for some constant $C>0$ and $A\simeq B$ if $A\lesssim B$ and $B\lesssim A$. We write $A\lesssim_n B$ whenever the implied constant depends on some parameter $n$. We omit such dependencies when they are of no importance.

In this paper we use several differentiation bases which are basically collections of open sets in $\R^n$. We use the symbol $\mathfrak B$ to denote a generic basis consisting of convex sets, the symbol $\mathfrak G$ to denote a generic basis consisting of rectangles, the symbol $\mathfrak R$ for the basis of all rectangles with sides parallel to the coordinate axes, the symbol $\mathfrak b$ for the collection of all Euclidean balls and the symbol $\mathfrak Q$ for the collection of all  cubes in $\R^n$ with sides parallel to the coordinate axes.

For a rectangle $R\in\mathfrak R$ we denote by $\mathcal D_R$ the mesh of ``dyadic rectangles'' associated to $R$. The ``dyadic children'' of $R$ are produced by dividing each side of $R$ into two equal parts while the dyadic parent of $R$ is the rectangle $R^{(1)}$ whose sidelengths are double the corresponding sidelengths of $R$ and shares exactly one corner with $R$. Thus every $R\subset \R^n$ has exactly $2^n$ dyadic children and is contained in a unique dyadic parent. For a dyadic rectangle $R$ we write $R^{(1)}$ for the parent of $R$ and $R^{(j)}$ for the ancestor of $R$ which is $j$ generations ``before'' $R$. For more details see the discussion before Proposition~\ref{p.dyadicmaximal}.

The measures $\mu,\nu$ that appear in this paper are always assumed to be locally finite, non-negative Borel measures. The symbol $\nu$ is used to denote such a measure in the underlying space. Thus we prove bounds on $L^p(\nu)$. We use the symbol $\mu$ to denote a measure which is additionally assumed to be \emph{doubling} with respect to some differentiation basis $\mathfrak B$. The doubling constant of $\mu$ with respect to $\mathfrak B$ is denoted by $\Delta_{\mu,\mathfrak B}$. We omit these indices when the definition of the underling measure or basis is clear from the context. The measure $\mu$ is typically used in order to define some maximal operator $M_{\mathfrak B,\mu}$. Thus our main questions concern the mapping property $M_{\mathfrak B,\mu}:L^p(\nu)\to L^p(\nu)$. We use the symbol $w$ to denote a non-negative locally integrable function, that is, a weight. In this case we write $M_{\mathfrak B,w}$ for the weighted maximal function with respect to $\mathfrak B$ and $w$.

Some words are also necessary concerning the dilations that we use in the paper. There are three kinds of dilations of a set $E$ with respect to some parameter $c>0$. If the set $E$ has a natural center of symmetry then $cE$ denotes the dilates of $E$ by a factor $c$, with respect to its center. If $B$ is a convex set we write $cB$ even if $B$ is not symmetric with respect to some point. In this case the homothecy center is taken to be the center of the John ellipsoid of $B$. See \S~\ref{s.john}. We write $\dil_cE\coloneqq \{cx:x\in E\}$, that is, for the dilation with respect to $0$. Finally we write $c*R$ whenever $R$ is a ``dyadic rectangle'' to denote the dilation of $R$ by the factor $c$, with respect to the corner shared by $R$ and its parent $R^{(1)}$. We also consider translations of sets; for $\sigma \in \R^n$ and $E\subset \R^n$ we set $\tau_\sigma E\coloneqq \{x+\sigma:x\in E\}$.

Finally, almost all the logarithms that appear in this paper are base-$2$ logarithms; we still just write $\log t$ for $\log_2 t$.

%%%%%%%%%%%%%%%%%%%%%%%%%%%%%% SECTION  SECTION SECTION
\section{Definitions and overview of known results} \label{s.overview}

%%%%%%%%%%%%%%%%%%%%%%%%%%%%%% SECTION  SECTION SECTION
\subsection{Bases of open sets and maximal operators} By a \emph{basis} $\mathfrak B$ we mean a collection of bounded open sets in $\R^n$. The differentiation properties of a basis $\mathfrak B$ are determined by the boundedness properties of the corresponding maximal function, acting on locally integrable functions $f$ as
\begin{align*}
	M_{\mathfrak B}f(x)\coloneqq \sup_{ \substack{B\in\mathfrak B \\ B\ni x}} \frac{1}{|B|}\int_B |f(y)| dy,
\end{align*}
if $x\in\cup_{B\in\mathfrak B}B$ and $M_{\mathfrak B}f(x)\coloneqq 0$ otherwise. Particular attention will be given to two special bases of open sets; namely the basis  $\mathfrak Q$, consisting of all $n$-dimensional cubes with sides parallel to the coordinate axes, and the basis $\mathfrak R$  consisting of all rectangles with sides parallel to the coordinate axes. The corresponding maximal operators are the \emph{Hardy-Littlewood maximal function} $M_{\mathfrak Q}$ and the \emph{strong maximal function}  $M_{\mathfrak R}$. We are interested in $L^p$-bounds for the maximal functions $M_{\mathfrak B}$ of the type
\begin{align}
	\label{e.basic} \|M_{\mathfrak B} f\|_{L^p(\R^n)}\lesssim_{\mathfrak B, p,n} \|f\|_{L^p(\R^n)},\quad 1<p\leq +\infty,
\end{align}
together with appropriate endpoint bounds as $p\to 1^+$. The existence of such bounds cannot be guaranteed in the generality of $\mathfrak B$ discussed above. Indeed, if $\mathfrak B$ is the family of \emph{all} rectangles in $\R^n$, allowing all rotations, dilations and translations, then $M_{\mathfrak B}$ is called the \emph{universal maximal function} which is known to be unbounded for any $p<+\infty$; see \cite{Guz}. Note, however, that this operator restricted to radial functions is bounded on  $L^p(\R^n)$ for $p>n$. See \cite{CHS} and \cite{DN}. On the other hand, the maximal operators defined with respect to the bases $\mathfrak R$ and $\mathfrak Q$ are well understood, with the corresponding sharp endpoint bounds being
\begin{align*}
 \Abs{\{x\in\R^n:M_{\mathfrak Q}f(x)>\lambda \}}&\lesssim_{n}\int_{\R^n} \frac{|f(x)|}{\lambda}dx,\quad \lambda>0,
\\
 \Abs{\{x\in\R^n:M_{\mathfrak R}f(x)>\lambda \}}&\lesssim_{n} \int_{\R^n} \Phi_n\Big(\frac{|f(x)|}{\lambda}\Big)dx,\quad \lambda>0.
\end{align*}
Here $\Phi_n(t)\coloneqq t(1+(\log^+t)^{n-1})$ and $\log^+t\coloneqq \max(\log t,0)$. The first weak inequality above is the classical maximal theorem of Hardy and Littlewood; see for example \cite{Stein}. The second distributional inequality is the strong maximal theorem of Jessen, Marcinkiewicz and Zygmund from \cite{JMZ}. See also \cite{CF} for a geometric approach to the same result. By interpolation, the previous endpoint bounds imply \eqref{e.basic} for both $\mathfrak Q$ and $\mathfrak R$.

%%%%%%%%%%%%%%%%%%%%%%%%%%%%%% SECTION  SECTION SECTION
\subsection{Weights associated to bases} We say that $w$ is a \emph{weight} associated to the basis $\mathfrak B$ if $w$ is a non-negative locally integrable function on $\mathbb R^n$ and $w(B)\coloneqq \int_B w(x)dx<+\infty$ for every $B\in\mathfrak B$. The weighted analogue of estimate \eqref{e.basic} takes the form
\begin{align}\label{e.basicw}
  \|M_{\mathfrak B} f\|_{L^p(w)}\lesssim_{\mathfrak B, p,n} \|f\|_{L^p(w)},\quad 1<p\leq +\infty.
\end{align}
The corresponding endpoint bounds as $p\to 1^+$ are also of great interest and are typically harder (and stronger) than their $L^p$-analogues \eqref{e.basicw}.

Again, for the bases $\mathfrak Q$ and $\mathfrak R$, estimates \eqref{e.basicw} are much better understood and the validity of \eqref{e.basicw} for any $1<p<+\infty$ is characterized by the membership of $w$ to the classes of Muckenhoupt weights $A_{p,\mathfrak B}$:

%%%%%%%%%%%%%%%%%%%%%%%%%%%%%% DEFINITION DEFINITION DEFINITION
\begin{definition}\label{d.Ap}
We say that a weight $w$ belongs to the class $A_{p,\mathfrak B}$, $1<p<+\infty$, if
\begin{align*}
	[w]_{A_{p,\mathfrak B}}\coloneqq \sup_{B\in\mathfrak B}\bigg(	\frac{1}{|B|}\int_B w(y)dy\bigg)\bigg(\frac{1}{|B|}\int_B w(y)^{1-p'}dy\bigg)^{p-1}<+\infty.
\end{align*}
Here and throughout the paper $p'$ denotes the dual exponent of $p$, that is $\frac{1}{p}+\frac{1}{p'}=1.$

For the limiting case $p=1$ the class $A_{1,\mathfrak B}$ is defined to be the set of weights $w$ such that
\begin{align*}
[w]_{A_{1,\mathfrak B}}\coloneqq	\sup_{B\in\mathfrak B}\bigg(\frac{1}{|B|}\int_B w(y)dy \bigg) \esssup_{B} (w^{-1})<+\infty.
\end{align*}
This is equivalent to $w$ having the property
\begin{align*}
	M_{\mathfrak B} w(x)\leq [w]_{A_{1,\mathfrak B}}\cdot w(x),\quad\text{a.e.}\ x\in\R^n.
\end{align*}
It follows from H\"older's inequality and the definitions above that for all $1\leq p< q<+\infty$ we have that $A_{p,\mathfrak B}\subset A_{q,\mathfrak B}$, that is, the classes $A_{p,\mathfrak B}$ are increasing in $p\geq 1$. It is thus natural to define the limiting class $A_{\infty,\mathfrak B}$ as
\begin{align*}
A_{\infty,\mathfrak B}\coloneqq \bigcup_{p> 1} 	A_{p,\mathfrak B}=\bigcup_{p\geq 1} 	A_{p,\mathfrak B}.
\end{align*}
For the special bases $\mathfrak Q,\mathfrak R$ we use the shorthand notation $A_p\coloneqq A_{p,\mathfrak Q}$ and $A_p ^*\coloneqq A_{p,\mathfrak R}$.
\end{definition}
%%%%%%%%%%%%%%%%%%%%%%%%%%%%%% DEFINITION DEFINITION DEFINITION

%%%%%%%%%%%%%%%%%%%%%%%%%%%%%% REMARK REMARK REMARK
\begin{remark} For a general basis $\mathfrak B$, as considered here, there is really no ``obvious'' definition of the class $A_{\infty,\mathfrak B}$. For the basis $\mathfrak Q$ many definitions have appeared in the literature and they are all equivalent to each other. See for example \cite{GaRu}. This remains true for the basis $\mathfrak R$. However, for a general basis $\mathfrak B$, these different definitions may define different classes of weights. We adhere to the definition of the class $A_{\infty,\mathfrak B}$ as the union of the classes $A_{p,\infty}$ for notational simplicity mostly, keeping in mind that for the bases $\mathfrak Q$ and $\mathfrak R$ the definition above coincides with all the standard definitions that appear in the literature.
\end{remark}
%%%%%%%%%%%%%%%%%%%%%%%%%%%%%% REMARK REMARK REMARK

For the bases $\mathfrak Q$ and $\mathfrak R$, the boundedness properties of the corresponding maximal operators on weighted Lebesgue spaces are well known. This is completely classical and due to Muckenhoupt for the basis $\mathfrak{Q}$; see \cite{Muck}. The theorem of Muckenhoupt extends without difficulty to the basis $\mathfrak R$ whenever $p>1$; see for example \cite{BaKu}. We summarize these results below.

%%%%%%%%%%%%%%%%%%%%%%%%%%%%%% THEOREM THEOREM THEOREM
\begin{theorem}\label{t.complete} The following statements are true.
\begin{itemize}
 \item[(i)] Let $\mathfrak B$ be either $\mathfrak Q$ or $\mathfrak R$ and $1<p<+\infty$. Then $M_{\mathfrak B}:L^p(w)\to L^p(w)$ if and only if $M_{\mathfrak B}:L^p(w)\to L^{p,\infty}(w)$, if and only if $w\in A_{p,\mathfrak B}$.
\item[(ii)] For $\mathfrak Q$: $M_{\mathfrak Q}:L^1(w)\to L^{1,\infty}(w)$ if and only if $w\in A_{1,\mathfrak Q}$.
\item[(iii)] For $\mathfrak R$: If $w\in A_{1,\mathfrak R}$ then
\begin{align}\label{e.endRw}
  w(\{x\in\R^n:M_{\mathfrak R}f(x)>\lambda \}) \lesssim_{n,w} \int_{\R^n} \Phi_n \Big(\frac{|f(x)|}{\lambda}\Big) w(x)dx,\quad \lambda>0.
\end{align}
\end{itemize}
\end{theorem}
%%%%%%%%%%%%%%%%%%%%%%%%%%%%%% THEOREM THEOREM THEOREM

Thus, the boundedness properties of $M_{\mathfrak B}$ on $L^p(w)$, $1<p<+\infty$, are completely characterized for the bases $\mathfrak {Q,R}$, and the same is true for the weighted endpoint estimate $M_\mathfrak Q:L^1(w)\to L^{1,\infty}(w)$. However, the endpoint estimate \eqref{e.endRw} for the strong maximal function is not so transparent. Indeed, the presence of the logarithmic terms on the right hand side of the \eqref{e.endRw} results in the condition $A_{1,\mathfrak R}$ being sufficient, but not necessary, for the validity of \eqref{e.endRw}. A necessary condition for \eqref{e.endRw}, which is weaker than $w\in A_{1,\mathfrak R}$, appears in \cite{BaKu} but the authors show that it is not sufficient. On the other hand, the weighted endpoint estimate \eqref{e.endRw} has been characterized in \cite{Gog} in terms of a certain covering property for rectangles. See also \cite{KK}*{Theorem 4.3.1} for a detailed proof of this fact. Similar characterizations of the boundedness properties of maximal operators on general $L^p(\mu)$-spaces in terms of covering properties are contained for example in \cite{J}*{Theorem 2.2} and \cite{GaRu}*{Lemma IV.6.11}, while the approach goes back to \cite{Cor} and \cite{CF}. There is however no characterization in the spirit of the Muckenhoupt $A_{p,\mathfrak R}$-classes, of the weights $w$ such that \eqref{e.endRw} holds.

%%%%%%%%%%%%%%%%%%%%%%%%%%%%%% SECTION  SECTION SECTION
\subsection{Maximal operators with respect to measures}\label{s.generalmeasures} Let $\mu$ be a non-negative measure on $\R^n$, finite on compact sets, and let $\mathfrak B$ be a basis. For $f\in L^1 _{\operatorname{loc}}(\mu)$ we write
\begin{align*}
 M  _{\mathfrak B,\mu}f(x)\coloneqq \sup_{ \substack{B\in\mathfrak B \\ B\ni x\\ \mu(B)>0}} \frac{1}{\mu(B)}\int_B |f(y)| d\mu(y),
\end{align*}
if $x\in\cup_{B\in\mathfrak B}B$ and $M _{\mathfrak B,\mu}f(x)\coloneqq 0$ if $x\notin \cup_{B\in\mathfrak B}B$. If $d\mu(x)=w(x)dx$ for some weight $w$ associated to the basis $\mathfrak B$ we just write $M_{\mathfrak B,w}$ for $M_{\mathfrak B,\mu}$ and this operator will be called the weighted maximal operator with respect to $w$. The boundedness properties of $M_{\mathfrak B,\mu}$ are much harder than the corresponding properties of the unweighted maximal operator, with definitive information only for special cases of bases $\mathfrak B$ and measures $\mu$. Again, we mainly restrict our attention to the case that $\mathfrak B$ is either $\mathfrak Q$ or $\mathfrak R$. As in the unweighted case, the one-parameter operator $M_{\mathfrak Q,\mu}$ is easier to analyze than the operator $M_{\mathfrak R,\mu}$. However, even in the one-parameter case, there is no complete characterization of the measures $\mu$ for which $M_{\mathfrak Q,\mu}$ is bounded on $L^p(\R^n,\mu)$. Below we give a brief overview of the known results for the bases $\mathfrak Q$ or $\mathfrak R$ and refer the interested reader to the monograph \cite{Jou} for further details and proofs.

%%%%%%%%%%%%%%%%%%%%%%%%%%%%%% SECTION  SECTION SECTION
\subsubsection{A special one-dimensional result}\label{s.one} In dimension $n=1$, let $\mu$ be any non-negative Borel measure. We have that $M_{\mathfrak Q,\mu}:L^1(\mu)\to L^{1,\infty}(\mu)$ and by interpolation $M_{\mathfrak Q,\mu}:L^p(\mu)\to L^p(\mu)$ for all $1<p\leq +\infty$. This result is very special to one dimension since the proof depends on a covering lemma for intervals of the real line. Observe that there is essentially no restriction on the measure $\mu$. See for example \cite{Sjo} for the details of this result.

%%%%%%%%%%%%%%%%%%%%%%%%%%%%%% SECTION  SECTION SECTION
\subsubsection{The centered, one-parameter maximal function with respect to a measure \texorpdfstring{$\mu$}{mu}.}\label{s.cent} A common variation of $M_{\mathfrak Q,\mu}$ is the \emph{centered} weighted Hardy-Littlewood maximal function, given as
\begin{align*}
 M_{\mathfrak Q,\mu} ^ {\operatorname{c}} f(x)\coloneqq \sup_{r>0} \frac{1}{\mu(Q(x,r))}\int_{Q(x,r)} |f(y)|d\mu(y),
\end{align*}
where $Q(x,r)$ denotes the cube with sides parallel to the coordinate axes and sidelength $r>0$, centered at $x\in\R^n$. Then for any non-negative Borel measure $\mu$ we have that $M_{\mathfrak Q,\mu} ^ {\operatorname{c}} :L^1(\R^n)\to L^{1,\infty}(\R^n)$ and thus, by interpolation, $M_{\mathfrak Q,\mu} ^ {\operatorname{c}} :L^p(\R^n)\to L^{p}(\R^n)$ for all $1<p\leq +\infty$.

The proof of this result depends on the Besicovitch covering lemma and it remains valid whenever the Besicovitch argument goes through. Thus, the condition that the maximal function defined above is centered is essential. For example, it was shown in \cite{Sjo} that if $\gamma$ is the Gaussian measure in $\R^2$ then the non-centered weighted maximal operator $M_{\mathfrak Q,\gamma}$ does not map $L^1$ to $L^{1,\infty}$.

The second essential hypothesis, hidden in the definition of $M_{\mathfrak Q,\mu} ^ {\operatorname{c}}$, is that it is a \emph{one-parameter} maximal operator, that is, we average with respect to a one-parameter family of cubes. Here one could replace cubes by Euclidean balls or more general ``balls'', given by translations and \emph{one-parameter} dilations of a  convex set in $\mathbb R^n$ symmetric about the origin.

On the other hand, emphasizing the need for the one-parameter hypothesis mentioned previously, the  boundedness fails for the weighted strong maximal function, even in its centered version. The reason is that the family $\mathfrak R$ is an $n$-parameter family of sets for which the Besicovitch covering is not valid. See for example \cite{F} for an example of a locally finite measure $\mu$ for which $M_{\mathfrak R,\mu} ^ {\operatorname{c}}$ is unbounded on $L^p(\mu)$ for all $p<\infty$.

%%%%%%%%%%%%%%%%%%%%%%%%%%%%%% SECTION  SECTION SECTION
\subsubsection{The non-centered, one-parameter maximal function with respect to a doubling measure}\label{s.ncent} Let $\mu$ be a non-negative Borel measure. The following definition is standard.
%%%%%%%%%%%%%%%%%%%%%%%%%%%%%% DEFINITION DEFINITION DEFINITION
\begin{definition} The measure $\mu$ is called \emph{doubling} if there is a constant $\Delta_\mu>0$ such that, for every cube $Q=Q(x,r)\subseteq \R^n$ we have
\begin{align*} \mu(2Q)\leq \Delta_\mu \mu(Q),\end{align*}
where $2Q=Q(x,2r)$.
\end{definition}
%%%%%%%%%%%%%%%%%%%%%%%%%%%%%% DEFINITION DEFINITION DEFINITION

It is an easy observation that for $\mu$ doubling, the non-centered weighted maximal operator $M_{\mathfrak Q,\mu}$ is pointwise equivalent to its centered version, that is, $M_{\mathfrak Q,\mu}   f(x)\simeq M_{\mathfrak Q,\mu} ^ {\operatorname{c}} f(x)$, where the implicit constants depend only on $\Delta_\mu$. It follows from the discussion in \S \ref{s.cent} that the maximal operator $M_{\mathfrak Q,\mu}$ with respect to a doubling measure $\mu$ maps $L^1(\mu)$ to $L^{1,\infty}(\mu)$ and $L^p(\mu)$ to $L^p(\mu)$ for all $1<p\leq \infty$.

Here, the doubling hypothesis cannot be removed. Indeed, the example from \cite{Sjo} mentioned above shows that there exists a non-doubling measure, in particular the Gaussian measure in $\R^2$, such that $M_{\mathfrak Q,\gamma}$ does not map $L^1$ to $L^{1,\infty}$. Most of the results in the literature that study $M_{\mathfrak B,\mu}$ for non-doubling measures concern the basis $\mathfrak b$ consisting of all \emph{Euclidean balls} in $\R^n$ and radial measures. For example it is shown in \cite{Vargas} that if $\mu$ is rotationally invariant and assigns positive measure to all open sets then $M_{\mathfrak b,\mu}:L^1(\mu)\to L^{1,\infty}(\mu)$ if and only if $\mu$ is \emph{doubling away from the origin}. In \cite{SjoSo} the authors provide a sufficient condition on a radial measure $\mu$ so that $M_{\mathfrak b,\mu}$ satisfies certain weak type inequalities close to $L^1(\mu)$ which in turn imply that $M_{\mathfrak b,\mu}:L^p(\mu)\to L^p(\mu)$.  An example of a radial measure $\mu$ such that $M_{\mathfrak b,\mu}$ is unbounded on all $L^p(\mu)$-spaces,  $p<+\infty$, can be found in \cite{Inf}.
Note that in the non-doubling case, the operators $M_{\mathfrak Q,\mu}$ and $M_{\mathfrak b,\mu}$ can behave quite differently, unlike the doubling case. For example, if $\mu$ is a product of non-negative one-dimensional Borel measures then obviously $M_{\mathfrak Q,\mu}\leq M_{\mathfrak R,\mu}$. By the one-dimensional result mentioned in \S \ref{s.one} and the methods from \cite{CaFa} we get that $M_{\mathfrak R,\mu}$, and a fortiori $M_{\mathfrak Q,\mu}$, is bounded on $L^p(\mu)$ and satisfies the endpoint estimate
\begin{align*}
 \mu(\{x\in\R^n:M_{\mathfrak Q,\mu}f(x)>\lambda\}) \leq \mu(\{x\in\R^n:M_{\mathfrak R,\mu}f(x)>\lambda\}) \lesssim_n \int_{\R^n} \Phi_n\Big(\frac{|f(x)|}{\lambda}\Big)d\mu(x).
\end{align*}
One such product measure is the Gaussian measure $\gamma$ on $\R^n$ for which we can thus conclude
\begin{align}\label{e.cubenondoubl}
   \gamma(\{x\in\R^n:M_{\mathfrak Q,\gamma}f(x)>\lambda \}) \lesssim_{n} \int_{\R^n} \frac{|f(x)|}{\lambda}\bigg(1+\Big(\log^+\frac{|f(x)|}{\lambda}\Big)^{n-1}\bigg)\gamma(x)dx.
\end{align}
On the other hand it was shown in \cite{SjoSo} that
\begin{align*}
   \gamma(\{x\in\R^n:M_{\mathfrak b,\gamma}f(x)>\lambda \}) \lesssim_{n} \int_{\R^n} \frac{|f(x)|}{\lambda}\bigg(1+\Big(\log^+\frac{|f(x)|}{\lambda}\Big)^\frac{n+1}{2}\bigg)\gamma(x)dx
\end{align*}
and that this is actually sharp. It is not currently known whether estimate \eqref{e.cubenondoubl} is sharp. The previous discussion shows however that, surprisingly, the operator $M_{\mathfrak R,\gamma}$ on $\R^2$ exhibits stronger endpoint behavior than $M_{\mathfrak b,\gamma}$.

%%%%%%%%%%%%%%%%%%%%%%%%%%%%%% SECTION  SECTION SECTION
\subsubsection{The weighted strong maximal function} The rectangle case is less understood. Even if the measure $\mu$ is doubling, it is not known in general whether $M_{\mathfrak R,\mu}$ maps $L^p(\R^n,\mu)$ to $L^p(\R^n,\mu)$. Notable exceptions are the one-dimensional case as well as the case where $\mu$ is a product measure, as we have already seen in \S \ref{s.ncent}. The case $d\mu(x)=w(x)dx$ has been studied more systematically, the main result being that of R. Fefferman from \cite{F}: If $w\in A_{\infty,\mathfrak R}$ then $M_{\mathfrak R,w}:L^p(w)\to L^p(w)$ for all $1<p\leq \infty$. The endpoint inequality
\begin{align*}
   w(\{x\in\R^n:M_{\mathfrak R,w}f(x)>\lambda \})&\lesssim_{n,w} \int_{\R^n} \Phi_n\Big(\frac{|f(x)|}{\lambda}\Big) w(x)dx
\end{align*}
is also true whenever $w\in A_{\infty,\mathfrak R}$. This was proved by Jawerth and Torchinsky in~\cite{JT} and independently by Long and Shen \cite{LOSH}. A weaker sufficient condition for the boundedness of $M_{\mathfrak R,w}$ appears in \cite{JT} but it is quite technical and we will not describe here; it shows however that $w\in A_{\infty,\mathfrak R}$ is \emph{not} a necessary condition for the boundedness of $M_{\mathfrak R,w}$ on $L^p(w)$, nor for the corresponding endpoint distributional estimate above.

%%%%%%%%%%%%%%%%%%%%%%%%%%%%%% SECTION  SECTION SECTION
\subsubsection{The weighted maximal function associated to a general basis} For general bases $\mathfrak B$ and associated weights $w$, very little is known concerning the boundedness of $M_{\mathfrak B}$ and $M_{\mathfrak B,w}$ on $L^p(w)$. However, the following abstract theorem from \cite{Per1} gives a necessary and sufficient condition for the boundedness of the $M_{\mathfrak B,w}$, in terms of the unweighted maximal function $M_{\mathfrak B}$ in the special case that $w\in A_{\infty,\mathfrak B}$.

%%%%%%%%%%%%%%%%%%%%%%%%%%%%%% THEOREM THEOREM THEOREM
\begin{theorem}[C. P\'erez]\label{t.perez} Let $\mathfrak B$ be a basis. The following are equivalent:
	\begin{itemize}
		\item[(i)] For every $1<p<+\infty$ and every $w\in A_{p,\mathfrak B}$ we have that
		\begin{align*}
			M_{\mathfrak B}:L^p(w)\to L^p(w).
		\end{align*}
		\item[(ii)] For every $1<p<+\infty$ and every $w\in A_{\infty,\mathfrak B}$ we have that
		\begin{align*}
			M_{\mathfrak B,w}:L^p(w)\to L^p(w).
		\end{align*}
	\end{itemize}
\end{theorem}
%%%%%%%%%%%%%%%%%%%%%%%%%%%%%% THEOREM THEOREM THEOREM

The previous theorem as well as Theorem \ref{t.complete} motivate the following definition, which is also from \cite{Per1}.

%%%%%%%%%%%%%%%%%%%%%%%%%%%%%% DEFINITION DEFINITION DEFINITION
\begin{definition} A basis $\mathfrak B$ is a \emph{Muckenhoupt basis} if for every $1<p<+\infty$ and every $w\in A_{p,\mathfrak B}$, we have that
	\begin{align*}
		M_{\mathfrak B}:L^p(w)\to L^p(w).
	\end{align*}	
\end{definition}
%%%%%%%%%%%%%%%%%%%%%%%%%%%%%% DEFINITION DEFINITION DEFINITION

With this definition Theorem~\ref{t.perez} states that $\mathfrak B$ is a Muckenhoupt basis if and only if the weighted maximal function satisfies $M_{\mathfrak B,w}:L^p(w)\to L^p(w)$ for every $1<p<+\infty$ and every $w\in A_{\infty,\mathfrak B}$. On the other hand, Theorem \ref{t.complete} shows that both $\mathfrak Q$ and $\mathfrak R$ are Muckenhoupt bases. Another interesting example of Muckenhoupt basis is given by the C\'ordoba-Zygmund basis in $\R^3$, consisting of rectangles with sides parallel to the coordinate axes and sidelengths of the form $(s,t,st), \ s,t>0$. See \cite{F2} and \cite{PiFe} for this and related facts. An example of a basis that is \emph{not} a Muckenhoupt basis is the collection of all rectangles in $\R^n$.   To see this, observe that the Lebesgue measure satisfies the $A_p$ condition in Definition~\ref{d.Ap} with respect to this basis for all $1 < p <+ \infty$, although the corresponding universal maximal operator is unbounded on $L^p(\R^n)$ for $1 < p <+ \infty$.

%%%%%%%%%%%%%%%%%%%%%%%%%%%%%% SECTION  SECTION SECTION
\section{Tauberian conditions and Muckenhoupt weights} The purpose of this section is to give a new characterization of the class of Muckenhoupt weights $A_{\mathfrak B,\infty}$ in the case that $\mathfrak B$ is a homothecy invariant basis consisting of convex sets. This is the content of Theorem \ref{t.taubw} and Corollary \ref{c.taubw} below. The main interest is in the case of the basis $\mathfrak R$ consisting of all rectangles in $\R^n$ with sides parallel to the coordinate axes. We show that $A_\infty ^* = A_{\infty,\mathfrak R}$ coincides with the class of weights $w$ satisfying a weighted Tauberian condition which, until now, had been considered weaker than $A_\infty ^*$.

%%%%%%%%%%%%%%%%%%%%%%%%%%%%%% SECTION  SECTION SECTION
\subsection{Tauberian conditions for homothecy invariant bases} In this paragraph we are concerned with bases $\mathfrak B$ that are \emph{homothecy invariant}, namely they satisfy
\begin{enumerate}
		\item[(i)] For every $B\in\mathfrak B$ and every $y\in\mathbb R^n$ we have that $\tau_y B \in\mathfrak B$, where $\tau_y B=\{x+y:x\in B\}$.
		\item[(ii)] For every $B\in\mathfrak B$ and $s>0$ we have that $\dil_sB \in\mathfrak B$, where $\dil_s B=\{sx:x\in B\}$.
\end{enumerate}

%%%%%%%%%%%%%%%%%%%%%%%%%%%%%% REMARK REMARK REMARK
\begin{remark} An example of a Muckenhoupt basis which is not homothecy invariant is provided by the basis $\mathfrak B_0=\{(0,b):b>0\}$  consisting of intervals on the real line. It is shown in \cite{DuoMO} that $\mathfrak B_0$ is a Muckenhoupt basis and clearly it is not homothecy invariant.
\end{remark}
%%%%%%%%%%%%%%%%%%%%%%%%%%%%%% REMARK REMARK REMARK

We will say that the corresponding maximal operator $M_{\mathfrak B}$ satisfies a \emph{Tauberian condition} with respect to a fixed $\gamma\in(0,1)$ if there exists some constant $c_{\mathfrak B,\gamma}>0$ such that, for every measurable set $E\subset\R^n$, we have
\begin{align*}\label{e.taub}
\tag{$\operatorname{A}_{\mathfrak B,\gamma }$} \Abs{\{x\in\R^n: M_{\mathfrak B}(\ind_E)(x)>\gamma\}}\leq c_{\mathfrak B,\gamma}  |E|.
\end{align*}
It is essential to notice here that the previous estimate is supposed to hold only for a fixed $\gamma\in(0,1)$. However, in practice, many times one has a Tauberian condition of the form \eqref{e.taub} for every $\gamma\in(0,1)$ and typically $c_{\mathfrak B,\gamma}$ blows up to infinity as $\gamma\to 0^+$.

This condition has appeared in different contexts in several places but we are mainly interested in the investigations initiated in \cite{HS}. There are a couple of situations where such a condition can arise naturally. For example if the operator $M_{\mathfrak B}$ is known to be of weak type $(p,p)$ for some $1\leq p<+\infty$ then $M_{\mathfrak B}$ satisfies \eqref{e.taub} with respect to every $\gamma \in(0,1)$. More interestingly, the Tauberian condition appears naturally when one considers \emph{density bases}.

%%%%%%%%%%%%%%%%%%%%%%%%%%%%%% DEFINITION DEFINITION DEFINITION
\begin{definition} Let $\mathfrak B$ be a homothecy invariant basis in $\R^n$. Then $\mathfrak B$ is called a \emph{density basis} if it differentiates $L^\infty(\R^n)$, namely if for every $f\in L^\infty(\R^n)$, for almost every $x\in\R^n$ and for every sequence $\{B_k\}_k \subset \mathfrak B$ such that $B_k\ni x$ and $\diam(B_k)\to 0$ we have
\begin{align*}
 \lim_{k\to+\infty} \frac{1}{|B_k|}\int_{B_k} f(y)dy =f(x).
\end{align*}
\end{definition}
%%%%%%%%%%%%%%%%%%%%%%%%%%%%%% DEFINITION DEFINITION DEFINITION

Note that since translations and dilations of all sets in $\mathfrak B$ are still in $\mathfrak B$ there is always a sequence of sets $\{B_k\}_k$ as in the previous definition. The relevance of density bases to Tauberian conditions is revealed by the following theorem of Busemann and Feller \cite{BF}; see also \cite{Guzdif}*{Chapter III, Theorem 1.2}:

%%%%%%%%%%%%%%%%%%%%%%%%%%%%%% THEOREM THEOREM THEOREM
\begin{theorem}[Busemann, Feller]\label{t.bf} Let $\mathfrak B$ be a homothecy invariant basis and $M_{\mathfrak B}$ be the corresponding maximal operator. Then $\mathfrak B$ is a density basis if and only if, for every $\gamma\in (0,1)$ there is a constant $0<c(\gamma)<+\infty$ such that, for every measurable set $E\subset \R^n$ we have
 \begin{align*}
  \Abs{\{x\in\R^n:M_{\mathfrak B}(\ind_E)(x)>\gamma\}}\leq c(\gamma)|E|.
 \end{align*}
\end{theorem}
%%%%%%%%%%%%%%%%%%%%%%%%%%%%%% THEOREM THEOREM THEOREM

Thus for homothecy invariant density bases $\mathfrak B$ consisting of convex sets, the corresponding maximal operator satisfies Tauberian conditions of the type \eqref{e.taub} with respect to \emph{every} $\gamma\in(0,1)$. The same is true if the operator $M_{\mathfrak B}$ is known to be of weak type $(p,p)$ for some $p\geq 1$. Even though this is apparently stronger than a single Tauberian condition with respect to a fixed $\gamma\in(0,1)$, the following striking theorem was proved in \cite{HS}:

%%%%%%%%%%%%%%%%%%%%%%%%%%%%%% THEOREM THEOREM THEOREM
\begin{theorem}[Hagelstein, Stokolos]\label{t.stokhag} Let $\mathfrak B$ be a homothecy invariant basis consisting of convex sets in $\R^n$. Then the following are equivalent:
\begin{enumerate}
 \item[(i)] The operator $M_{\mathfrak B}$ satisfies a Tauberian condition~\eqref{e.taub} with respect to some \emph{fixed} $\gamma\in(0,1)$.
 \item[(ii)] There exists some $1<p_o=p_o(\mathfrak B,\gamma,n)<+\infty$ such that $M_{\mathfrak B}:L^p(\R^n)\to L^p(\R^n)$ for all $p>p_o$.
\end{enumerate}
\end{theorem}
%%%%%%%%%%%%%%%%%%%%%%%%%%%%%% THEOREM THEOREM THEOREM

In virtue of the previous theorem, the Tauberian condition \eqref{e.taub} for a single $\gamma\in(0,1)$ is equivalent to Tauberian conditions for \emph{every} $\gamma\in(0,1)$ whenever $\mathfrak B$ is a homothecy invariant basis consisting of convex sets in $\R^n$.

%%%%%%%%%%%%%%%%%%%%%%%%%%%%%% SECTION  SECTION SECTION
\subsection{Weighted Tauberian conditions} We now come to the main subject of the current paper which investigates variations of the Tauberian condition \eqref{e.taub} under the presence of a weight $w$ or, somewhat more generally, a measure $\mu$. For this section we still consider the \emph{unweighted} maximal operator $M_{\mathfrak B}$ acting on the weighted space $L^p(w)$. We will be mostly interested in the case that $\mathfrak B$ is either $\mathfrak Q$ or $\mathfrak R$, but we will see that our results remain valid for general homothecy invariant bases consisting of convex sets. We thus fix a  basis $\mathfrak B$ and a weight $w$ associated to $\mathfrak B$.

%%%%%%%%%%%%%%%%%%%%%%%%%%%%%% DEFINITION DEFINITION DEFINITION
\begin{definition}
We will say that the maximal operator $M_{\mathfrak B}$ satisfies a weighted Tauberian condition with respect to some $\gamma\in(0,1)$ and a weight $w$ if there exists a constant $c_{\mathfrak B, \gamma,w} >0$ such that, for all measurable sets $E\subset \R^n$ we have
\begin{align*}\label{e.taubw}
\tag{$\operatorname{A}_{\mathfrak B,\gamma,w}$}	 w(\{x\in\R^n: M_{\mathfrak B}(\ind_E)(x)>\gamma\})\leq c_{\mathfrak B, \gamma,w} \, w(E).
\end{align*}
\end{definition}
%%%%%%%%%%%%%%%%%%%%%%%%%%%%%% DEFINITION DEFINITION DEFINITION

The weighted Tauberian condition has been considered many times in the literature, especially in the context of weighted inequalities for the strong maximal function. Indeed, it appears for example in \cite{GLPT}, \cite{J}, \cite{JT}, \cite{LL}, \cite{Per1} and \cite{Per}. Condition \eqref{e.taubw} is typically presented in the literature as a presumably weaker substitute for the hypothesis $w\in A_{\infty, \mathfrak B}$ and is usually referred to as condition $(\operatorname{A})$. For example, in \cite{JT} condition \eqref{e.taubw} is used as a hypothesis in order to prove $L^p(w)$-bounds for the weighted strong maximal function $M_{\mathfrak R,w}$. Likewise, in \cite{Per1} it is shown that if $\mathfrak B$ is a Muckenhoupt basis and $M_{\mathfrak B}$ satisfies \eqref{e.taubw} for a fixed $\gamma\in(0,1)$ then $M_{\mathfrak B}$ satisfies a Fefferman-Stein inequality, namely
\begin{align*}
 \int_{\R^n} M_{\mathfrak B}f(x) ^p w(x)dx\lesssim_{n,p,w} \int_{\R^n} |f(x)|^p M_{\mathfrak B}w(x)dx,\quad 1<p<+\infty.
\end{align*}
Finally, in  \cite{GLPT} and  \cite{LL}, the condition is used in order to deal with covering properties of rectangles which are relevant in the study of two weight problems for the strong maximal function. The following theorem shows however that condition \eqref{e.taubw} is just an equivalent characterization of $A_{\infty,\mathfrak B}$ for quite a large class of bases $\mathfrak B$. This is the content of our first main result:

%%%%%%%%%%%%%%%%%%%%%%%%%%%%%% THEOREM THEOREM THEOREM
\begin{theorem}\label{t.taubw} Let $\mathfrak B$ be a homothecy invariant basis consisting of convex sets. Let $w$ be a non-negative, locally integrable function on $\R^n$. Then the following are equivalent:
\begin{itemize}
 \item [(i)] Condition \eqref{e.taubw} is satisfied for the weight $w$ and the basis $\mathfrak B$, for a fixed level $\gamma\in(0,1)$.
 \item [(ii)] There exists  $1<p_o=p_o(c_{\mathfrak B, \gamma,w},\gamma,n)<+\infty$ such that $M_{\mathfrak B}:L^p(w)\to L^p(w)$ for all $p>p_o$.
\end{itemize}
\end{theorem}
%%%%%%%%%%%%%%%%%%%%%%%%%%%%%% THEOREM THEOREM THEOREM

%%%%%%%%%%%%%%%%%%%%%%%%%%%%%% PROOF PROOF PROOF
\begin{proof} It is trivial that (ii) implies (i). The converse implication for $w\equiv 1$ is essentially \cite{HS}*{Theorem 1}. Now a careful inspection of the proofs of the relevant results in \cite{HS} reveals that the arguments therein actually show that if $M_\mathfrak B$ and $w$ satisfy \eqref{e.taubw} then $M_{\mathfrak B}$ is of restricted type $(q,q)$ for some $q>1$, with respect to the weight $w$. Marcinkiewicz interpolation now gives (ii) for any $p>q$. Alternatively, the proof follows from the more general result in the current paper, namely Theorem~\ref{t.convdoubling}.
\end{proof}
%%%%%%%%%%%%%%%%%%%%%%%%%%%%%% PROOF PROOF PROOF

%%%%%%%%%%%%%%%%%%%%%%%%%%%%%% REMARK REMARK REMARK
\begin{remark}
 In fact the same proof goes through to show that if $\mu$ is a non-negative Borel measure which is finite on compact sets and $\mu(\{M_{\mathfrak B}(1_E)>\gamma\})\leq c_{\mathfrak B, \gamma,\mu} \mu(E)$ for every measurable set $E$ then $M_{\mathfrak B}:L^p(\mu)\to L^p(\mu)$ for some $p>1$. We will not insist on the this generalization here as Theorem~\ref{t.convdoubling} below provides a much more general statement.
\end{remark}
%%%%%%%%%%%%%%%%%%%%%%%%%%%%%% REMARK REMARK REMARK
If the basis $\mathfrak B$ is additionally a Muckenhoupt basis we immediately get the following corollary:

%%%%%%%%%%%%%%%%%%%%%%%%%%%%%% COROLLARY COROLLARY COROLLARY
\begin{corollary}\label{c.taubw} Let $\mathfrak B$ be a homothecy invariant basis consisting of convex sets which is also a Muckenhoupt basis. Then $w\in A_{\infty,\mathfrak B}$ if and only if $w$ satisfies \eqref{e.taubw} for \emph{some fixed level} $\gamma\in(0,1)$.
\end{corollary}
%%%%%%%%%%%%%%%%%%%%%%%%%%%%%% COROLLARY COROLLARY COROLLARY

Specializing to the case $\mathfrak B=\mathfrak Q$ or $\mathfrak B=\mathfrak R$, Corollary~\ref{c.taubw} provides a new characterization of the usual Muckenhoupt class $A_\infty$ and the strong Muckenhoupt class $A_\infty ^*$, respectively. However, condition \eqref{e.taubw} is a lot harder to check in practice than all the other equivalent formulations of $A_\infty$ and $A_\infty ^*$ that appear in the literature. Our main point however is that the assumption \eqref{e.taubw} for $\mathfrak B=\mathfrak R$ is exactly the same as $w\in A_\infty^*$ and not weaker, as it was so far believed.

%%%%%%%%%%%%%%%%%%%%%%%%%%%%%% SECTION  SECTION SECTION
\section{Doubling measures with respect to general bases} Consider the maximal function $M_{\mathfrak B,\mu}$ defined with respect to a non-negative measure $\mu$ which is finite on compact sets and a homothecy invariant basis $\mathfrak B$ in $\R^n$, consisting of open and bounded convex sets with non-empty interior. Our main objective is to find a characterization of the measures $\mu$ such that $M_{\mathfrak B,\mu}:L^p(\nu)\to L^p(\nu)$ for some $p>1$ in terms of a mixed $\mu,\nu$-Tauberian condition. To make this precise we give the following definition:

%%%%%%%%%%%%%%%%%%%%%%%%%%%%%% DEFINITION DEFINITION DEFINITION
\begin{definition}
We will say that the maximal operator $M_{\mathfrak B,\mu}$ satisfies the Tauberian condition \eqref{e.taubconv} with respect to some \emph{fixed} $\gamma \in(0,1)$ if there is a constant $c_{\mathfrak B,\gamma,\nu} ^\mu >0$ such that for all measurable sets $E\subset\R^n$ we have
\begin{align*}\label{e.taubconv}
 \tag{$\operatorname{A}_{\mathfrak B,\gamma,\nu} ^\mu$} \nu(\{x\in\R^n:M_{\mathfrak B,\mu}(\ind_E)(x)>\gamma\})\leq c_{\mathfrak B,\gamma,\nu} ^\mu \nu(E).
\end{align*}
\end{definition}
%%%%%%%%%%%%%%%%%%%%%%%%%%%%%% DEFINITION DEFINITION DEFINITION

If $M_{\mathfrak B,\mu}:L^p(\nu)\to L^p(\nu)$ for some $p>1$ then \eqref{e.taubconv} is satisfied for all $\gamma\in(0,1)$. We want to investigate if the converse is true, namely, if \eqref{e.taubconv} for some fixed $\gamma\in(0,1)$ implies the boundedness of $M_{\mathfrak B,\mu}$ on $L^p(\nu)$ for sufficiently large $p>1$. We will not pursue this in full generality but rather confine ourselves to the case that the measure $\mu$ is doubling with respect to the underlying basis.

%%%%%%%%%%%%%%%%%%%%%%%%%%%%%% SECTION  SECTION SECTION
\subsection{Doubling measures with respect to a differentiation basis}\label{s.gendoubling}
In this section we discuss the properties of measures which are doubling with respect to a basis of convex sets. By a basis of convex sets we will always mean a homothecy invariant basis $\mathfrak B$ consisting of non-empty, bounded, open convex sets with non-empty interior.

We begin by the definition of doubling in this general context. Remember that for $\sigma\in\R^n$, $E\subset \R^n$ and $c>0$ we write $\tau_\sigma E= \{\sigma+x: x\in E\}$ and $\dil_c E=\{cx: x\in E\}$.

%%%%%%%%%%%%%%%%%%%%%%%%%%%%%% DEFINITION DEFINITION DEFINITION
\begin{definition} Let $\mu$ be a non-negative measure which is finite on compact sets. We will say that $\mu$ is doubling with respect to $\mathfrak B$ if there is a constant $ \Delta_{\mu,\mathfrak B}>1$ such that, for every $B\in \mathfrak B$ and every $\sigma \in \R^n$ such that $B\subset \tau_\sigma \dil_2 B $ we have
\begin{align*}
 \mu(\tau_\sigma \dil_2 B)\leq \Delta_{\mu,\mathfrak B} \mu(B).
\end{align*}
We always assume $\Delta_{\mu,\mathfrak B}$ to be the smallest possible constant so that the previous inequality holds uniformly for all $B\in\mathfrak B$. When the underlying basis $\mathfrak B$ is clear from the context we will write $\Delta_{\mu}$ for $\Delta_{\mu,\mathfrak B}$.
\end{definition}
%%%%%%%%%%%%%%%%%%%%%%%%%%%%%% DEFINITION DEFINITION DEFINITION

%%%%%%%%%%%%%%%%%%%%%%%%%%%%%% REMARK REMARK REMARK
\begin{remark}
The previous definition of a doubling measure reduces to the usual doubling condition (up to changes in the doubling constant) if $\mathfrak B=\mathfrak Q$ or $\mathfrak B=\mathfrak b$. However, the doubling condition with respect to $\mathfrak R$, say, is quite different than the doubling condition with respect to cubes. In fact, if one wants to study the behavior of the operator $M_{\mathfrak R,\mu}$ with respect to a measure $\mu$ then the ``natural'' condition is that $\mu$ is doubling with respect to $\mathfrak R$. For example in \cite{PWX}, measures that are doubling with respect to $\mathfrak R$ are called \emph{product-doubling} and we will adopt the same terminology here. The same notion of product-doubling is discussed for example in \cite{JT}. Naturally, weights $w\in A_\infty ^*$ give product-doubling measures $w(x)dx$.

Observe also that for a general basis of convex sets $\mathfrak B$ there is in general no natural homothecy center as the convex sets in $\mathfrak B$ might not be symmetric with respect to some point. In order to avoid confusion in all these subtle issues we will always specify the basis according to which a measure is assumed to be doubling.
\end{remark}
%%%%%%%%%%%%%%%%%%%%%%%%%%%%%% REMARK REMARK REMARK

\subsection{The John ellipsoid}\label{s.john} One of the technical annoyances when dealing with general convex sets is the lack of a natural homothecy center as the convex sets we will consider will not in general be symmetric with respect to some point. In order to deal with this lack of symmetry and resulting technical issues, the classical lemma of F. John, \cite{John}, will be very useful. See also \cite{Ball} for a very nice exposition of this and related topics.

%%%%%%%%%%%%%%%%%%%%%%%%%%%%%% LEMMA LEMMA LEMMA
\begin{lemma}[F. John] Let $B$ be a bounded convex set in $\R^n$. Then $B$ contains a unique ellipsoid $\mathcal E_B$, of maximal volume. We will call $\mathcal E_B$ the \emph{John ellipsoid} of $B$. The John ellipsoid of $B$ is such that
\begin{align*}
 \mathcal E_B \subset B \subset n \mathcal E_B.
\end{align*}
Here $c\mathcal E_B$ denotes the dilation of the ellipsoid $\mathcal E_B$ by a factor $c>0$ with respect to its center.
\end{lemma}
%%%%%%%%%%%%%%%%%%%%%%%%%%%%%% LEMMA LEMMA LEMMA

Given a basis $\mathfrak B$ consisting of convex sets we will now construct an associated basis $ \mathfrak G_{\mathfrak B}$ consisting of rectangles as in \cite{HS}. To this end let $B\in\mathfrak B$ and $\mathcal E_B$ be the John ellipsoid of $B$. Then there is a (not necessarily unique) rectangle $R \supset \mathcal E_B$ of minimal volume. It is elementary to check that for any ellipsoid $\mathcal E$, a rectangle $R$ of minimal volume that contains $\mathcal E$ satisfies
\begin{align*}
 \mathcal E \subset R \subset \sqrt n \mathcal E.
\end{align*}
Given $B \in \mathcal B$, let $R_{B}$ be a rectangle of minimal volume containing $n\mathcal E_{B}$.  By the above observations and John's lemma we get for every $B\in \mathcal B$ that
$$
  B  \subset n \mathcal E_B \subset R_B\subset  n\sqrt{n}\mathcal E_{B} \subset n^{3/2}B .
$$

Here the dilations $cB$ are with respect to the center of the John ellipsoid associated to $B$.   We now define the basis $\mathfrak G_{\mathfrak B}$ as
\begin{align*}
 \mathfrak G_{\mathfrak B} \coloneqq\{ R_B:B\in\mathfrak B\}.
\end{align*}
Since $\mathfrak B$ is homothecy invariant the rectangle $R_{B}$ may be selected so that $\mathfrak{G_{\mathfrak B}}$ is homothecy invariant. In this paper we will always assume that this is the case.

The following lemma is an immediate consequence of the above discussion. We omit the easy proof.

%%%%%%%%%%%%%%%%%%%%%%%%%%%%%% LEMMA LEMMA LEMMA
\begin{lemma}\label{l.convrect} Let $\mathfrak B$ be a basis of convex sets and $ \mathfrak G_{\mathfrak B}$ be the homothecy invariant basis of associated rectangles  as constructed above. Suppose that $\mu$ is doubling with respect to $\mathfrak B$ with doubling constant $\Delta_{\mu,\mathfrak B}$. We have:
\begin{itemize}
 \item [(i)] The measure $\mu$ is doubling with respect to $\mathfrak G_{\mathfrak B}$ with doubling constant
\begin{align*}
\Delta_{\mu,\mathfrak G_{\mathfrak B} }\leq \Delta_{\mu,\mathfrak B} ^{ 1+\lceil\frac{3}{2}  \log  n  \rceil } .
\end{align*}
\item [(ii)] We have the pointwise equivalence
\begin{align*}
 \frac{1}{c_n}  M_{\mathfrak G_{\mathfrak B},\mu}f(x) \leq M_{\mathfrak B,\mu}f(x)\leq  c_n   M_{\mathfrak G_{\mathfrak B},\mu}f(x),\quad x\in\R^n,
\end{align*}
where $c_n\coloneqq \Delta_{\mu,\mathfrak B} ^{ \lceil  \frac{3}{2 }\log  n   \rceil } $.
\item[(iii)] If $B\in\mathfrak B$ and $R_B$ is the associated rectangle of $B$ with $B\subset R_B\subset n^\frac{3}{2}B$ then
\begin{align*}
	\mu(B)\geq  \rho \mu(R_B),
\end{align*}
where $\rho\coloneqq c_n ^{-1}$ and $c_n$ as defined in (ii).
\end{itemize}
\end{lemma}
%%%%%%%%%%%%%%%%%%%%%%%%%%%%%% LEMMA LEMMA LEMMA

\subsection{Properties of general doubling measures} The doubling condition has some important consequences in that the measure is ``homogeneously'' distributed in the space. We summarize these properties in the proposition below. We note that these properties are classical and refer the reader to \cite{Stein}*{\S 8.6} for more details.

%%%%%%%%%%%%%%%%%%%%%%%%%%%%%% PROPOSITION PROPOSITION PROPOSITION
\begin{proposition}\label{p.noboundarymass} Let $\mu$ be a (not identically zero) locally finite, non-negative Borel measure. Assume that $\mu$ is doubling with respect to some family $\mathfrak K$ consisting of all the homothetic copies of a fixed rectangle. The following properties are satisfied.
\begin{itemize}
\item [(i)]  We have $\mu(U)>0$ for every open set $U\subset \R^n$.
\item [(ii)] Let $R\in \mathfrak K$ and $\mathcal D_R$ be the dyadic grid generated by $R$. There exists a constant $\gamma_\mu>1$, depending only on the doubling constant of $\mu$ and the dimension $n$ such that $\mu(R)\leq \gamma_\mu ^{-m} \mu(R^{(m)})$, where $R^{(m)}$ is the ancestor of $R$, $m$ generations higher. In particular $\mu(\R^n)=+\infty$.
\item [(iii)] The maximal operator $M_{\mathfrak K,\mu}$ is of weak type $(1,1)$ and strong type $(p,p)$  for all $1<p\leq \infty$, with respect to $\mu$, and the operator norms depend  only on the doubling constant of the measure $\mu$, the exponent $p$ and the dimension $n$. Also the centered maximal operator $M_{\mathfrak K,\mu} ^{\operatorname{c}}$  satisfies the same bounds.
\item[(iv)] If $B$ is a convex set in $\R^n$ we have $\mu(\partial B)=0$ where $\partial B\coloneqq \bar B\setminus B$ is the boundary of $B$.
\end{itemize}
\end{proposition}
%%%%%%%%%%%%%%%%%%%%%%%%%%%%%% PROPOSITION PROPOSITION PROPOSITION

%%%%%%%%%%%%%%%%%%%%%%%%%%%%%% PROOF PROOF PROOF
\begin{proof} The proof of (i) can be found for example in \cite{Stein}*{\S 8.6}. For (ii) let $R^{(1)}$ be the dyadic parent of $R$ and let $\{R_j\}_{j=1} ^{2^n}$ denote the dyadic children of $R^{(1)}$ and suppose that $R=R_1$. Then
	\begin{align*}
		\mu(R^{(1)})&=\sum_{j=1} ^{2^n}\mu(R_j) =\mu(R_1) +\sum_{j=2} ^{2^n}\mu(R_j)\geq (1+(2^n-1)\delta_\mu ^{-1})\mu(R_1),
	\end{align*}
where $\delta_\mu>1$ is the doubling constant of $\mu$. Let $\gamma_\mu=1+(2^n-1)\delta_\mu ^{-1}>1$. Since $R$ is $m$ generations inside $R^{(m)}$ we iterate to get $\mu(R)\leq \gamma_\mu ^{-m}\mu(R^{(m)})$ as desired.
	
For (iii) observe that $M_{\mathfrak K,\mu}$ is essentially the Hardy-Littlewood maximal operator with respect to a doubling measure and the result is classical. Since the measure $\mu$ is doubling the operators $M_{\mathfrak K,\mu}, M_{\mathfrak K,\mu} ^{\operatorname{c}}$ are pointwise comparable and satisfy the same bounds.

Finally for (iv) let us fix the convex set $B$ and $x\in \partial B$. Let $H$ be a supporting hyperplane of $B$ through $x$ and let $H^-$ be the open half-space defined by $H$ so that $H^-\cap B=\emptyset$. Let $R\in\mathfrak K$, centered at $x$ and $sR$ be the rectangle with the same center as $R$ and sides $s<1$ times the corresponding sides of $R$. So $sR$ is an homothetic copy of $R$. Consider the $4^n$ subrectangles $R_{s,j}$ produced by dividing each side of $sR$ into four equal parts. Now at least one of these $R_{s,j}$'s is contained in the open half space $H^-$. Let us call this rectangle $R'$ and observe that it is of the form $R'=z+\frac{1}{4}sR \subset sR$ and $R'\cap \overline{B}=\emptyset$. We can then estimate
\[
\begin{split}
 \mu(\partial B\cap sR )&=\mu(\partial B\cap sR \cap R' )+\mu(\partial B \cap sR \setminus R')
\\
&=\mu(\partial B \cap sR  \setminus R')\leq \mu(sR )-\mu(R')
\\
&\leq \mu(sR)-\frac{1}{\delta_\mu ^2} \mu(sR )\leq c \mu(sR ),
\end{split}
\]
with $c<1$. Applying (iii) for the centered operator $M_{\mathfrak K,\mu} ^{\operatorname{c}}$ we see that
\begin{align*}
 1>c\geq \mu(\partial B\cap sR )/\mu(sR )\to \ind_{\partial B},\, \mu\text{-almost everywhere as} \ s\to 0^+,
\end{align*}
which implies that $\mu(\partial B)=0$.
\end{proof}
%%%%%%%%%%%%%%%%%%%%%%%%%%%%%% PROOF PROOF PROOF

%%%%%%%%%%%%%%%%%%%%%%%%%%%%%% SECTION  SECTION SECTION
\section{Tauberian conditions for bases of rectangles} We now turn our attention to the maximal function $M_{\mathfrak G,\mu}$ defined with respect to a non-negative measure $\mu$ which is finite on compact sets and a homothecy invariant basis $\mathfrak G$ in $\R^n$ consisting of rectangles. Observe that we do not assume the rectangles in $\mathfrak G$ to have sides parallel to the coordinate axes but one possible choice of $\mathfrak G$ is the basis $\mathfrak R$.

Our main objective is to find a characterization of the measures $\mu$ such that $M_{\mathfrak G,\mu}:L^p(\nu)\to L^p(\nu)$ for some $p>1$ in terms of a mixed $\mu,\nu$-Tauberian condition. The Tauberian condition \eqref{e.taubconv} now takes the form

\begin{align*}\label{e.taubmunu}
 \tag{$\operatorname{A}_{\mathfrak G,\gamma,\nu} ^\mu$} \nu(\{x\in\R^n:M_{\mathfrak G,\mu}(\ind_E)(x)>\gamma\})\leq c_{\mathfrak G,\gamma,\nu} ^\mu \nu(E).
\end{align*}

Our second main result gives a characterization of the boundedness of $M_{\mathfrak G,\mu}$ on $L^p(\nu)$ in terms of the Tauberian condition \eqref{e.taubmunu}, whenever the measure $\mu$ is doubling with respect to $\mathfrak G$. Note that for the measure $\nu$ we only assume that it is non-negative and locally finite.

%%%%%%%%%%%%%%%%%%%%%%%%%%%%%% THEOREM THEOREM THEOREM
\begin{theorem}\label{t.doubling} Let $\mathfrak G$ be a homothecy invariant basis consisting of rectangles and $\mu,\nu$ be two non-negative measures on $\R^n$, finite on compact sets. Assume that $\mu$ is doubling with respect to $\mathfrak G$. The following are equivalent:
\begin{itemize}
 \item [(i)] The measures $\mu,\nu$ satisfy the Tauberian condition \eqref{e.taubmunu} with respect to some \emph{fixed level $\gamma\in(0,1)$}.
 \item [(ii)] There exists $1<p_o= p_o(c_{\mathfrak G,\gamma,\nu} ^\mu ,\gamma, \mu)<+\infty$ such that $M_{\mathfrak G,\mu}:L^p(\nu)\to L^p(\nu)$ for all $p>p_o$.
\end{itemize}
\end{theorem}
%%%%%%%%%%%%%%%%%%%%%%%%%%%%%% THEOREM THEOREM THEOREM
The previous theorem has an interesting corollary whenever $\mu\equiv \nu$. In this special case our main theorem concerns the boundedness of the operator $M_{\mathfrak G,\mu}$ on $L^p(\mu)$, for sufficiently large $p>1$ and $\mu $ doubling with respect to $\mathfrak G$. As discussed in \S~\ref{s.generalmeasures} this scenario is very well understood for the basis $\mathfrak Q$. Indeed, we already know that for a doubling measure $\mu$ the operator $M_{\mathfrak Q,\mu}$ is of weak type $(1,1)$ and thus of strong type $(p,p)$ for all $p>1$. Thus both (i) and (ii) of this theorem are always satisfied for $\mathfrak Q$ and $\mu\equiv\nu$. However, for $\mathfrak G=\mathfrak R$ and $\mu$ product-doubling we get a new characterization of the measures $\mu$ such that $M_{\mathfrak R,\mu}$ is bounded on $L^p(\mu)$, for sufficiently large $p>1$.

When $\mu\equiv \nu$ the mixed Tauberian condition becomes:
\begin{align*}\label{e.taubmu}
 \tag{$\operatorname{A}_{\mathfrak G,\gamma,\mu} ^\mu$} \mu(\{x\in\R^n:M_{\mathfrak G,\mu}(\ind_E)(x)>\gamma\})\leq c_{\mathfrak G,\gamma,\mu} ^\mu \mu(E).
\end{align*}
We then have:
%%%%%%%%%%%%%%%%%%%%%%%%%%%%%% THEOREM THEOREM THEOREM
\begin{corollary}\label{c.taubmu} Let $\mathfrak G$ be a homothecy invariant basis consisting of rectangles and $\mu$ be a non-negative measure on $\R^n$, finite on compact sets. Assume that $\mu$ is doubling with respect to $\mathfrak G$ . The following are equivalent:
\begin{itemize}
 \item [(i)] The measure $\mu$ satisfies the Tauberian condition \eqref{e.taubmu} with respect to some \emph{fixed level $\gamma\in(0,1)$}.
 \item [(ii)] There exists $1<p_o=p_o(c_{\mathfrak G,\gamma,\nu} ^\mu ,\gamma, \mu)<+\infty$ such that $M_{\mathfrak G,\mu}:L^p(\mu)\to L^p(\mu)$ for all $p>p_o$.
\end{itemize}
\end{corollary}
%%%%%%%%%%%%%%%%%%%%%%%%%%%%%% THEOREM THEOREM THEOREM

%%%%%%%%%%%%%%%%%%%%%%%%%%%%%% SECTION  SECTION SECTION
\subsection{Proof of Theorem~\ref{t.doubling}} In this subsection we give the details of the proof of Theorem~\ref{t.doubling}. First of all observe that if $M_{\mathfrak{G},\mu}: L^p(\nu)\to L^p(\nu)$ then trivially \eqref{e.taubmunu} is satisfied for every $\gamma\in(0,1)$. For the rest of this section we will thus assume that \eqref{e.taubmunu} holds for some $\gamma \in(0,1)$. Let $\beta\in(\gamma,1)$. Any such choice of $\beta$ will work equally well but for definitiveness we can take $\beta$ to be the arithmetic mean of $\gamma$ and $1 $. The hypothesis implies that
\begin{align}\label{e.amu'}
\nu (\{x\in\R^n: M_{\mathfrak G,\mu}(\ind_E)(x)\geq \beta  \})\leq  \mathbf{c}  \nu (E)\quad \text{for all measurable sets}\quad E\subseteq \mathbb R^n.
\end{align}
Here $\mathbf{c}=c_{\mathfrak G,\gamma,\nu} ^\mu$ but we suppress these dependencies for the sake of simplicity. We will need the following notation introduced in \cite{HS}. For every measurable set $E\subset \mathbb R^n$ we define $\mathcal{H}_{\beta}^0(E):=E$ and for $k\geq 1$
\begin{align*}
	\mathcal{H}_{\beta}^k(E):=\{x\in\R^n:M_{\mathfrak G,\mu} (\ind_{\mathcal{H}_{\beta}^{k-1}(E)})(x)\geq\beta\} .
\end{align*}
With these definitions at hand it is not difficult to check the following basic properties. Let $k,k'\geq 0$ be non-negative integers and $A,B$ measurable subsets of $\R^n$. Then
\begin{align}
&\mathcal{H}_{\beta}^1(\mathcal{H}_{\beta} ^{k}(A) )=\mathcal{H}_{\beta} ^{k+1}(A), \label{e.p1}
\\
&A\subseteq B\Rightarrow\mathcal{H}_{\beta} ^k(A)\subseteq{H}_{\beta} ^k(B),\label{e.p2}
\\
&\text{If } k'\leq k \text{ then }\mathcal H_{\beta} ^{k'}(A)\subseteq \mathcal H_{\beta} ^{k}(A).\label{e.p3}
\\
&\text{\eqref{e.taubmu} implies  \eqref{e.amu'} which in turn implies that } \nu (\mathcal H^k _\beta(A))\leq \mathbf{c} ^k \nu(A).\label{e.p4}
\end{align}
The properties above will be used in several parts of the proof with no particular mention.

The following lemma is the heart of the proof of Theorem~\ref{t.doubling}.

%%%%%%%%%%%%%%%%%%%%%%%%%%%%%% LEMMA LEMMA LEMMA
\begin{lemma}\label{l.main} Let $\mu$ be a doubling measure with respect to $\mathfrak G$, with doubling constant $\Delta_\mu$, and $E$ be a measurable set in $\mathbb R^n$. Suppose that for some $\alpha\in(0,\beta)$ and $R\in \mathfrak G$ we have $\frac{1}{\mu(R)}\int_R 1_E d\mu = \alpha$. Then
\begin{align*}
	R\subset\mathcal{H}_{\beta}^{k_{\alpha,\beta}}(E)\quad \text{where}\quad	 k_{\alpha,\beta}\coloneqq \bigg\lceil \frac{-\log(\frac{\beta}{\alpha})}{\log\beta}\bigg\rceil\bigg\lceil 2+\frac{\log^+(\beta\Delta_\mu ) )}{\log(1/\beta)}\bigg\rceil+1.
\end{align*}
Here we denote by $\lceil x \rceil$ the smallest positive integer which is no less than $x$.
\end{lemma}
%%%%%%%%%%%%%%%%%%%%%%%%%%%%%% LEMMA LEMMA LEMMA

Before giving the proof of the lemma let us see how we can use it to conclude the proof of Theorem~\ref{t.doubling}. By restricted weak type interpolation it suffices to show that for every $0<\lambda<1$ and every measurable set $E\subset \R^n$ we have the estimate
\begin{align}\label{e.main}
		\nu(\{x\in\R^n: M_{\mathfrak G,\mu}(\ind_E)(x)>\lambda \})\leq \frac{C}{\lambda^{p_o }} \nu(E)
\end{align}
for some $p_o>1$ and some constant $C>0$, independent of $\lambda$ and $E$. Estimate \eqref{e.main} above is the claim that the sublinear operator $M_{\mathfrak G,\mu}$ is of restricted weak type $(p_o ,p_o )$ with respect to the measure $\nu$, for some $p_o>1$.  Now we have
\begin{align}
\notag		\nu(\{x\in\R^n: M_{\mathfrak G,\mu}(\ind_E)(x)>\lambda \})&\leq	\nu (\{x\in\R^n: \lambda<M_{\mathfrak G,\mu}(\ind_E)(x)<\beta\})
		\\
	\label{e.interm}	& \quad +	\nu(\{x\in\R^n: M_{\mathfrak G,\mu}(\ind_E)(x)\geq \beta\})
		\\
\notag		&\leq 	\nu(\{x\in\R^n: \lambda<M_{\mathfrak G,\mu}(\ind_E)(x)<\beta\})+\frac{\mathbf{c}}{\lambda^{p_o}}\nu (E),
\end{align}
by \eqref{e.amu'}, for all $p_o>0$. In order to estimate the first summand  let $E_{\lambda,\beta}\coloneqq \{\lambda<M_{\mathfrak G,\mu}(\ind_E)(x)<\beta\}$. For every $x\in  E_{\lambda,\beta}$ there exists $R_x\in\mathfrak G$ and $\lambda<\alpha<\beta$ with
\begin{align*}
R_x\ni x,\  \mu(R_x)>0\quad \text{and} \quad \frac{\mu(R_x\cap E)}{\mu(R_x)} =\alpha.
\end{align*}
By Lemma \ref{l.main} we get that $R_x\subset \mathcal H_{\beta} ^{k_{\alpha,\beta}}(E)$. Now observe that $k_{\alpha,\beta}$ is a nonincreasing function of $\alpha$. Thus for all $\alpha>\lambda$ we have that $k_{\alpha,\beta}\leq k_{\lambda,\beta}$ which by \eqref{e.p3} implies that $\mathcal H_{\beta} ^{k_{\alpha,\beta}}(E)\subseteq \mathcal H_{\beta} ^{k_{\lambda,\beta}}(E)$. Combining these observations we get that
\begin{align*}
	E_{\lambda,\beta}\subseteq \bigcup_{x\in E_{\lambda,\beta}}R_x\subseteq \mathcal H_{\beta} ^{k_{\lambda,\beta}}(E).
\end{align*}
Using \eqref{e.p4} we now see that
\begin{align*}
	\nu(E_{\lambda,\beta}) \leq \nu (\mathcal H_{\beta} ^{k_{\lambda,\beta}}(E))\leq \mathbf{c}  ^{k_{\lambda,\beta}}\nu (E).
\end{align*}
By the explicit expression for $k_{\lambda,\beta}$ observe that we can write
\begin{align*}
	k_{\lambda,\beta}\leq \frac{\log(\frac{\beta}{\lambda})}{\log\frac{1}{\beta}} \eta_{\beta,\mu}+1
\end{align*}
with $\eta_{\beta,\mu}\geq 2$, depending only on $\beta$ and $\mu$. Thus
\begin{align*}
	\mathbf{c} ^{k_{\lambda,\beta}}\leq \mathbf{c} \mathbf{c} ^{\eta_{\beta,\mu}\log\frac{1}{\lambda} / \log \frac{1}{\beta}}\leq  \frac{\mathbf{c}}{\lambda^ {p_o} }=\frac{c_{\mathfrak G,\gamma,\nu} ^\mu}{\lambda^ {p_o} },
	\end{align*}
with $p_o =\eta_{\beta,\mu} \sfrac  {\log c_{\mathfrak G,\gamma,\nu} ^\mu }   { \log (1/\beta) }  >0$. Remember that $\beta$ is completely determined by the level $\gamma$ in hypothesis \eqref{e.taubmu} so that $p_o =p_o (c_{\mathfrak G,\gamma,\nu} ^\mu   ,\gamma,\mu )$.  Together with \eqref{e.interm} this completes the proof of \eqref{e.main} and thus of Theorem~\ref{t.doubling}.

For the proof of Lemma \ref{l.main} we will need an intermediate result. For this we introduce a final piece of notation. If $R\in \mathfrak G$ then there is a natural ``dyadic system of rectangles'' associated to $ R$ which we will denote by $\mathcal D _R$. This system has the properties
\begin{itemize}
	\item [(i)] We have that $R\in \mathcal D_R \subseteq \mathfrak G$.
	\item [(ii)] Every $S\in\mathcal D _R$ has a unique dyadic parent $S^{(1)}$ and $2^n$ dyadic children. Furthermore, each corner of a rectangle $S\in \mathcal D_R$ is shared by $S$ and exactly one of its dyadic children.
	\item [(iii)] If $V,S\in\mathcal D_R$ then $V\cap S\in \{\emptyset, V,S\}$.
\end{itemize}
We leave the details of the dyadic construction above to the interested reader. We now define the dyadic weighted maximal function with respect to $\mathcal D_R$ and $\mu$ as
\begin{align*}
	M_{\mathcal D_R,\mu} f(x):=\sup_{\substack{S\in\mathcal D _R\\
	S\ni x \\ \mu(S)>0}} \frac{1}{\mu(S)} \int_S |f(y)|d\mu(y),\quad x\in\mathbb R^n .
\end{align*}
The dyadic maximal function just defined satisfies all the desired bounds:

%%%%%%%%%%%%%%%%%%%%%%%%%%%%%% PROPOSITION PROPOSITION PROPOSITION
\begin{proposition}\label{p.dyadicmaximal} Let $\mu$ be a locally finite non-negative measure. We have that $M_{\mathcal D_R,\mu}  :L^1(\mu)\to L^{1,\infty}(\mu)$. We conclude that the family $\{R:R\in\mathcal D_R, R\ni x,\mu(R)>0 \}$ differentiates $L^1 _{\operatorname{loc}}(\mu)$.
\end{proposition}
%%%%%%%%%%%%%%%%%%%%%%%%%%%%%% PROPOSITION PROPOSITION PROPOSITION

Note that there is no doubling assumption on the measure $\mu$ in this proposition. Indeed, the proof amounts to selecting the maximal ``dyadic rectangles'' $S\in\mathcal D_R\cap[0,2^N)^n$ such that $\frac{1}{\mu(S)}\int_S |f(y)|d\mu(y)>\lambda$ and noting that they are disjoint. One then lets $N\to +\infty$. An identical argument  works for ``dyadic rectangles'' contained in the other quadrants of $\R^n$. We leave the details to the interested reader.

%%%%%%%%%%%%%%%%%%%%%%%%%%%%%% LEMMA LEMMA LEMMA
\begin{lemma}\label{l.interm} Let $\mu,E$ and $R$ be as in the hypothesis of Lemma \ref{l.main} above. Then there exists a non-negative integer $N$ such that
	\begin{align*}
		\mu(R\cap \mathcal H_\beta ^{N+2}(E))\geq \frac{1}{\beta}\mu(E\cap R).
	\end{align*}
\end{lemma}
%%%%%%%%%%%%%%%%%%%%%%%%%%%%%% LEMMA LEMMA LEMMA

%%%%%%%%%%%%%%%%%%%%%%%%%%%%%% PROOF PROOF PROOF
\begin{proof} We perform a Calder\'on-Zygmund decomposition of $\ind_{E\cap R}$ at level $\beta$ with respect to the dyadic grid $\mathcal D_R$. Namely, let $\{S_j\}_j\subset\mathcal D_R$ be the collection of ``dyadic rectangles'' which are maximal among the $S\in\mathcal D_R$ that satisfy
	\begin{align*}
		\frac{1}{\mu(S)}\int_S \ind_{E\cap R}(y)d\mu(y) >\beta.
	\end{align*}
Observe that $\mu(S)>0$ for all rectangles $S$ by Proposition~\ref{p.noboundarymass}. Furthermore $\mu(E\cap R)/\mu(R)<\beta $ so that every dyadic rectangle $S$ as above is contained in a maximal dyadic rectangle. This selection algorithm together with the hypothesis $\mu(R\cap E)/\mu(R)=\alpha<\beta$ allows us to choose a $\mu$-a.e. disjoint family $\{S_j\}_j\subset\mathcal D_R$ such that
\begin{align}
	& \bigcup_j S_j \subseteq R,\quad S_j\neq R\text{ for all }j,\notag
	\\
	& \{x\in\R^n: 	M_{\mathcal D_R ,\mu}  (\ind_{E\cap R})(x)>\beta\}=\bigcup_j S_j,\label{e.CZ2}	
	\\
	&  \frac{1}{\mu(S_j)}\int_{S_j}\ind_{E\cap R}d\mu>\beta ,\notag
	\\
	& \ind_{E\cap R}\leq \ind_{\cup_j S_j}\quad\mu\text{-a.e. in}\quad R.\label{e.CZ4}
\end{align}
For any constant $c>1$ we let $ c * S_j$ denote the rectangle containing $S_j$ that has sidelength $c$ times the sidelength of $S_j$ and has a common corner with $S_j$ and $S_j ^{(1)}$. With this notation we have $S_j ^{(1)}=2*S_j$ while the doubling hypothesis for $\mu$ implies that $\mu(S_j ^{(1)}) \leq \Delta_\mu  \mu(S_j)$ for every $j$.

For each $j$ we set $S_{j,0}\coloneqq S_j$. Suppose we have defined $S_{j,0}\subset \cdots \subset S_{j,k} $ for some $k\geq 0 $. We define $S_{j,k+1}$ to be a rectangle of the form $c_{j,k+1}*S_{j}$, where $c_{j,k+1}>1$ is chosen so that $S_{j,k}\subset S_{j,k+1}$ and
\begin{align}\label{e.incr}
 \frac{\mu(S_{j,k+1})}{\mu(S_{j,k})}=\frac{1}{\beta}>1.
\end{align}
Observe that such a choice is always possible since the function $f(c)\coloneqq \mu(c*S_{j,k})/\mu(S_{j,k})$ satisfies $f(1)=1$, $f(c)\to +\infty $ as $c\to +\infty$ and by (iv) of Proposition~\ref{p.noboundarymass} it is continuous on $[1,+\infty)$.

For  $k\geq 0$ we now set
	\begin{align*}
		E_k:=\bigcup_j S_{j,k}.
	\end{align*}
Observe that for $k\geq 0$ we have
	\begin{align}
		\label{e.claim} E_{k+1}\subset\{x\in\R^n:M_{\mathfrak G,\mu}(\ind_{E_k})(x)\geq \beta\}.
	\end{align}
Indeed if $x\in E_{k+1}$ then $x\in  S_{j_0,k+1}$ for some $j_0$. We estimate
\begin{align*}
M_{\mathfrak G,\mu}(\ind_{E_k})(x)&=\sup_{\substack{S\in\mathfrak G \\ S\ni x}}\frac{\mu(S\cap E_k)}{\mu(S)}\geq
\frac{\mu\big(S_{j_0,k+1}\cap\bigcup_j  S_{j,k}\big)}{\mu( S_{j_0,k+1})}  \geq \frac{\mu\big( S_{j_0,k}\big)}{\mu(  S_{j_0,k+1})}= \beta,
\end{align*}
by \eqref{e.incr}.  Next we claim that for every $k\geq 0$ we have
\begin{align}
	\label{e.E_H} E_k\subset\mathcal{H}_{\beta}^{k+1} (E).
\end{align}
For $k=0$ this is an immediate consequence of \eqref{e.CZ2} since
\begin{align*}
	E_0&=\bigcup_j S_j=\{x\in\R^n: M_{\mathcal D_R,\mu}  ( \ind_{E\cap R})(x)> \beta\}
	\\
	&\subseteq \{x\in\R^n: M_{\mathfrak G,\mu}(\ind_{E})(x)\geq \beta\}=\mathcal H_\beta ^1(E).
\end{align*}
Assume now that \eqref{e.E_H} is valid for some $k\geq 0$. By \eqref{e.claim}, the inductive hypothesis and properties \eqref{e.p1},\eqref{e.p2} we get that
	\begin{align*}
		E_{k+1}&\subset\{x\in\R^n:M_{\mathfrak G,\mu}(\ind_{E_{k}})(x)\geq\beta\}=\mathcal{H}_{\beta}^1(E_{k})
		\subseteq \mathcal{H}_{\beta}^1(\mathcal{H}_{\beta}^{k+1}(E))=\mathcal{H}_{\beta}^{k+2}(E),
	\end{align*}
which proves the claim.

Now let $N$ be the smallest non-negative integer such that $\beta^{-(N+1)}\geq \Delta_\mu $, where $\Delta_\mu$ is the doubling constant of the measure $\mu$. It follows that
\begin{align}\label{e.Nchoice}
 S_j ^{(1)} \subseteq  S_{j,N+1}
\end{align}
for every $j$. Indeed, assume for the sake of contradiction that $S_{j,N+1}\subsetneq S_j ^{(1)}$. Then the doubling property of $\mu$ implies that $\mu(S_{j,N+1})<\mu( S_j ^{(1)}).$ Thus
\begin{align*}
\Delta_\mu  \geq \frac{\mu(S_j ^{(1)})}{\mu(S_j)} >\frac{\mu(S_{j,N+1}  )}{\mu(S_j)} = {\beta}^{-(N+1)}
\end{align*}
which contradicts the choice of $N$.

Now \eqref{e.Nchoice} implies that for every $j$ we have
	\begin{align*}
\frac{\mu( S_{j,N})}{\mu\big(S_j^{(1)}\big)}&\geq \frac{\mu( S_{j,N})}{\mu(S_{j,N+1})}=\beta
	\end{align*}
and we can conclude that for every $j$
	\begin{align*}
		\frac{\mu(E_{N}\cap S_j^{(1)})}{\mu\big(S_j^{(1)}\big)}&=\frac{\mu\big(\bigcup_\nu S_{\nu ,N} \cap S_j^{(1)} \big)}{\mu\big(S_j^{(1)}\big)} \geq\frac{\mu\big( S_{j,N} \cap S_j^{(1)}\big) } {\mu\big(S_j^{(1)}\big)}\geq \min\bigg (1,\frac{\mu\big( S_{j,N} \big) } {\mu\big(S_j^{(1)}\big)}\bigg)\geq \beta.
	\end{align*}
	Hence
	\begin{align}
		\label{e.EQparent} \bigcup_j S_j^{(1)} \subseteq  \{x\in R: M_{\mathfrak G,\mu}(\ind_{E_{N }})(x)\geq\beta \}.
	\end{align}
	Let $\{S_{j_{k}} ^{(1)}\}_k$ denote the maximal elements of $\{S_j ^{(1)}\}_j$. Then the sets $\{S_{j_{k}} ^{(1)}\}_k$ are $\mu$-a.e. pairwise disjoint and $\bigcup_k S_{j_{k}} ^{(1)}=\bigcup_j S_j^{(1)}$. Note that all $S_{j_k}^{(1)}$'s are contained in $R$ since for all $j$ we have $S_j\subsetneqq R$. We also have that we have $S_{j_k} ^{(1)}\neq S_m$ for any $k,m$. Indeed, if $S_{j_k} ^{(1)}=S_m$ for some $k,m$ then we would have $S_{j_k} ^{(1)}\subsetneqq S_m ^{(1)}$ which is impossible because of the maximality of the $S_{j_k} ^{(1)}$'s among the $S_m ^{(1)}$'s. Thus none of the $S_{j_k} ^{(1)}$ were selected in the Calder\'on-Zygmund decomposition so that
	\begin{align*}
		\mu(S_{j_k} ^{(1)}\cap E\cap R )\leq \beta \mu(S_{j_k} ^{(1)})
	\end{align*}
and hence $ \mu(S_{j_k} ^{(1)}\cap E)\leq \beta \mu(S_{j_k} ^{(1)})$ for all $k$ since $S_{j_k} ^{(1)}\subseteq R$ for all $k$. Using the last estimate and \eqref{e.EQparent} we now have
	\begin{align*}
		\mu(\{x\in R: M_{\mathfrak G,\mu}(\ind_{E_{N}})(x)\geq\beta\})&\geq\mu\big( \bigcup_j S_j^{(1)}\big)=\mu\big( \bigcup_k S_{k_j}^{(1)}\big) \\
		&=\sum_k \mu(S_{k_j}^{(1)})\geq\frac{1}{\beta}\sum_k\mu(E\cap S_{k_j}^{(1)})\\
		&=\frac{1}{\beta}\mu(E\cap \bigcup_kS_{k_j}^{(1)})=\frac{1}{\beta}\mu(E\cap \bigcup_jS_{j}^{(1)})\\
		&\geq\frac{1}{\beta}\mu(E\cap \bigcup_jS_{j}).
	\end{align*}
Now \eqref{e.CZ4} implies that $\ind_{E\cap R}\leq \ind_{R\cap \cup_j S_j }$ almost everywhere so that $\mu(E\cap R)\leq \mu(R\cap \cup_j S_j)$. Thus the previous estimate reads
\begin{align*}
	\mu(\{x\in R: M_{\mathfrak G,\mu}(\ind_{E_{N}})(x)\geq\beta\})\geq \frac{1}{\beta}\mu(E\cap R)
\end{align*}
which by \eqref{e.E_H} implies that $\mu(R\cap \mathcal H_\beta ^{N+2}(E))\geq \beta^{-1}\mu(E\cap R)$ as desired.
\end{proof}
We can now conclude the proof of Lemma \ref{l.main}.

%%%%%%%%%%%%%%%%%%%%%%%%%%%%%% PROOF PROOF PROOF
\begin{proof}[Proof of Lemma \ref{l.main}] By the hypothesis of the lemma there exists $\alpha\in(0,\beta)$ and $R\in\mathfrak G$ with $\mu(E\cap R)/\mu(R)=\alpha$. Let $j_o$ be the smallest positive integer such that $\beta^{-j_o}\alpha \geq \beta$. Such an integer obviously exists since $\beta<1$. There are two possibilities.
	
\subsection*{case 1:} We have that $\mu(R\cap \mathcal H_\beta ^{j(N+2)}(E))< \beta \mu(R)$ for $j=0,\ldots,j_o-1$. Then we claim that we have
\begin{align}
	\label{e.induction} \mu(R\cap \mathcal H_\beta ^{k(N+2)}(E))\geq\frac{1}{\beta^k} \mu(R\cap E)\quad\text{for all}\quad k=1,\ldots,j_o.
\end{align}
We will prove \eqref{e.induction} by induction on $k$. Indeed, the case $k=1$ is just Lemma \ref{l.interm}. Assume that \eqref{e.induction} is true for some $1\leq k \leq j_o-1$. Then, since $\mu(R\cap \mathcal H_\beta ^{k(N+2)}(E))< \beta \mu(R)$ we can apply Lemma \ref{l.interm} for the rectangle $R$ and the set $H_\beta ^{k(N+2)}(E)$ in place of $E$ to conclude that
\begin{align*}
	\mu(R\cap \mathcal H_\beta ^{N+2} (\mathcal H_\beta ^{k(N+2)}(E)) )\geq \frac{1}{\beta}\mu\big(\mathcal H_\beta ^{k(N+2)}(E)\cap R\big)\geq\frac{1}{\beta}\big(\frac{1}{\beta}\big)^{k}\mu(R\cap E)=\big(\frac{1}{\beta}\big)^{k+1}\mu(R\cap E).
\end{align*}
However this is just \eqref{e.induction} for $k+1$ since $\mathcal H_\beta ^{N+2} (\mathcal H_\beta ^{k(N+2)}(E)) =\mathcal H_\beta ^{(k+1)(N+2)}(E)$.

Now by \eqref{e.induction} for $k=j_o$ we get that
\begin{align*}
	\frac{1}{\mu(R)}\mu\big(R\cap \mathcal H_\beta ^{j_o(N+2)}(E)\big)\geq \big(\frac{1}{\beta}\big)^{j_o}\frac{\mu(R\cap E)}{\mu(R)}=\beta^{-j_o}\alpha\geq \beta
\end{align*}
by the choice of $j_o$. This implies that $R\subseteq \mathcal H_\beta ^{j_o(N+2)+1}(E)$.

\subsection*{case 2:} We have that $\mu(R\cap \mathcal H_\beta ^{j(N+2)}(E))\geq \beta \mu(R)$ for some $j\in\{0,\ldots,j_o-1\}$. In fact, by the hypothesis we necessarily have that $j\geq 1$ in this case. Then
\begin{align*}
	\frac{1}{\mu(R)}\mu\big(R\cap \mathcal H_\beta ^{j(N+2)}(E) \big)\geq \beta
\end{align*}
which implies that $R\subseteq \mathcal \{x\in\R^n: M_{\mathfrak G,\mu}(\ind_{\mathcal H_\beta ^{j(N+2	)}(E) })(x)\geq \beta\}=\mathcal H_\beta ^{j(N+2)+1}(E)$.

Observe that in either one of the complementary cases considered above we can conclude that $R\subseteq \mathcal H_\beta ^{j_o(N+2)+1}(E)$. This proves the lemma with $k_{\alpha,\beta}=j_o(N+2)+1$. It remains to estimate $k_{\alpha,\beta}$. This can be easily done by going back to the way the integers $N$ and $j_o$ were chosen. For $N$ remember that it is the smallest non-negative integer such that $(1/\beta)^{N+1}\geq  \Delta_\mu $. If $1/\beta\geq  \Delta_\mu $ then the choice $N=0$ will do. If $1/\beta<\Delta_\mu $ then we get that $N$ is the smallest positive integer which is greater or equal to $\log(\beta \Delta_\mu  )/\log (1/\beta) $. Thus the choice
\begin{align*}
	N\coloneqq\bigg\lceil \frac{\log^+ (\beta \Delta_\mu ) } {\log (1/\beta) }\bigg\rceil
\end{align*}
covers both cases. Likewise, $j_o$ is the smallest integer such that $\beta^{-j_o}\geq \beta/\alpha$ or $j_o$ is the smallest integer greater than $\log(\beta/\alpha)/\log(1/\beta)$. Thus we can choose
\begin{align}
	j_o\coloneqq \bigg\lceil \frac{\log(\frac{\beta}{\alpha})}{\log\frac{1}{\beta}}\bigg\rceil.
\end{align}
We set
\begin{align*}
	k_{\alpha,\beta}&\coloneqq j_o (N+2)+1= \bigg\lceil \frac{\log(\frac{\beta}{\alpha})}{\log\frac{1}{\beta}}\bigg\rceil \bigg(\bigg\lceil \frac{\log^+(\beta\Delta_\mu )}{\log(1/\beta)}\bigg\rceil+2\bigg)+1
	\\
	&=\bigg\lceil \frac{\log(\frac{\beta}{\alpha})}{\log\frac{1}{\beta}}\bigg\rceil \bigg\lceil 2+ \frac{\log^+(\beta\Delta_\mu )}{\log(1/\beta)}\bigg\rceil+1.
\end{align*}
Of course, any integer greater than the $k_{\alpha,\beta}$ above will also do since the sets $\mathcal H_{\beta} ^k(E)$ are increasing in $k$.
\end{proof}
%%%%%%%%%%%%%%%%%%%%%%%%%%%%%% PROOF PROOF PROOF

%%%%%%%%%%%%%%%%%%%%%%%%%%%%%% SECTION  SECTION SECTION
\section{An extension to bases of convex sets}
The purpose of this section is to provide an extension of Theorem~\ref{t.doubling} to the case that the Tauberian condition is given with respect to a homothecy invariant basis $\mathfrak B$ consisting of convex sets:
\begin{align*}\label{e.convex}
 \tag{$\operatorname{A}_{\mathfrak B,\gamma,\nu} ^\mu$} \nu(\{x\in\R^n:M_{\mathfrak B,\mu}(\ind_E)(x)>\gamma\})\leq c_{\mathfrak B,\gamma,\nu} ^\mu \nu(E).
\end{align*}
As in the previous section where the basis $\mathfrak G$ consisted of rectangles, we will need to assume the doubling property of the measure $\mu$ with respect to the basis $\mathfrak B$. The main theorem of this section is the following.

%%%%%%%%%%%%%%%%%%%%%%%%%%%%%% THEOREM THEOREM THEOREM
\begin{theorem}\label{t.convdoubling} Let $\mathfrak B$ be a homothecy invariant basis consisting of convex sets and $\mu,\nu$ be two non-negative measures on $\R^n$, finite on compact sets. Assume that $\mu$ is doubling with respect to $\mathfrak B$. The following are equivalent:
\begin{itemize}
 \item [(i)] The measures $\mu,\nu$ satisfy the Tauberian condition \eqref{e.convex} with respect to some \emph{fixed level $\gamma\in(0,1)$}.
 \item [(ii)] There exists $1<p_o=p_o(c_{\mathfrak B,\gamma,\nu} ^\mu,n,\gamma,\mu)<+\infty$ such that $M_{\mathfrak B,\mu}:L^p(\nu)\to L^p(\nu)$ for all $p>p_o$.
\end{itemize}
\end{theorem}
%%%%%%%%%%%%%%%%%%%%%%%%%%%%%% THEOREM THEOREM THEOREM

The general strategy of the proof is the following. Assuming that \eqref{e.convex} is satisfied for some level $\gamma\in(0,1)$ we will show that the maximal operator $M_{\mathfrak G_{\mathfrak B},\mu}$ also satisfies a Tauberian condition with respect to every level $\alpha \in(\gamma,1)$. We will then use Theorem \ref{t.doubling} to conclude that $M_{\mathfrak G_{\mathfrak B},\mu}$ is bounded on some $L^p(\nu)$-space, for sufficiently large $p$. According to Lemma~\ref{l.convrect} the operators $M_{\mathfrak G_{\mathfrak B},\mu}$, $M_{\mathfrak B,\mu}$ are pointwise comparable so this will complete the proof of Theorem \ref{t.convdoubling}.

\subsection{The Tauberian condition for  \texorpdfstring{$M_{\mathfrak G_{\mathfrak B},\mu}$}{MGM}} In the subsection we will show that \eqref{e.convex} implies a Tauberian condition for the operator $M_{\mathfrak G_{\mathfrak B},\mu}$. This is the content of:

%%%%%%%%%%%%%%%%%%%%%%%%%%%%%% LEMMA LEMMA LEMMA
\begin{lemma}\label{l.taubassoc} Suppose that $M_{\mathfrak B,\mu}$ satisfies the Tauberian condition~\eqref{e.convex} for some fixed level $\gamma\in(0,1)$. Then for all $\alpha \in(\gamma,1) $ the operator $M_{\mathfrak G _{\mathfrak B},\mu}$ satisfies a Tauberian condition with respect to $\alpha$:
\begin{align*}\label{e.taubassoc}
 \tag{$\operatorname{A}_{\mathfrak G_{\mathfrak B},\alpha,\nu} ^\mu$} \nu(\{x\in\R^n:M_{\mathfrak B,\mu}(\ind_E)(x)>\alpha \})\leq c_{\mathfrak G_{\mathfrak B},\alpha ,\nu} ^\mu \nu(E)
\end{align*}
where $c_{\mathfrak G_{\mathfrak B},\alpha ,\nu} ^\mu$ depends on $c_{\mathfrak   B ,\gamma,\nu} ^\mu$, the measures $\mu,\nu$, the dimension $n$, $\gamma$ and $\alpha$.
\end{lemma}
%%%%%%%%%%%%%%%%%%%%%%%%%%%%%% LEMMA LEMMA LEMMA
The proof of this lemma is the most crucial step towards Theorem~\ref{t.convdoubling} and its proof will be explained through several intermediate steps in this section.

We will adopt the definitions and notation of \S~\ref{s.gendoubling}, namely for every $B\in\mathfrak B$ we consider the associated rectangle $R_B$ where $B \subset R_B \subset n^\frac{3}{2}B$ and $\mathfrak G_{\mathfrak B}=\{R_B:B\in\mathfrak B\}$ forms a homothecy invariant basis. Our basic assumption is that $\mu$ is doubling with respect to $\mathfrak B$ with doubling constant $ \Delta_{\mu,\mathfrak B}$. By Lemma~\ref{l.convrect} this implies that $\mu$ is also doubling with respect to $\mathfrak G_{\mathfrak B}$, with doubling constant $\Delta_{\mu,\mathfrak G_{\mathfrak B}}$. All these notions and constants will be fixed throughout this section so we will just write $\Delta_\mu \coloneqq \Delta_{\mu,\mathfrak B}$ and $\delta_\mu \coloneqq \Delta_{\mu,\mathfrak G_{\mathfrak B}}$.

We now fix a convex set $B$ and its associated rectangle $R=R_B\supset B$ and work locally inside $R$. By using a bijection $T:\R^n\to \R^n$ we always have $[0,1]^n=Q=T(R)$ and then we set $K\coloneqq T(B)\subset Q$. By considering the pushforward of $\mu$, that is the measure defined as $\mu_T(E)\coloneqq \mu(T^{-1}E)$ for every measurable set $E$, we readily see that the measure $\mu_T$ is doubling with respect to the basis $T\mathfrak B\coloneqq\{T(B):B\in\mathfrak B\}$ with doubling constant $\Delta_\mu$. Also the measure $\mu_T$ is doubling with respect to the basis $T\mathfrak G$ with doubling constant $\delta_\mu$. Using these invariances we can and will henceforth assume that $R=Q$ and $B$ is a convex set inside $Q$. We will use the same notation $\mu$ for the measure $\mu_T$. This will hopefully create no confusion as all our estimates will only depend on the doubling constants which are the same for both measures.

The following lemma is the heart of the matter when it comes to the proof of Lemma~\ref{l.taubassoc}.
%%%%%%%%%%%%%%%%%%%%%%%%%%%%%% LEMMA LEMMA LEMMA
\begin{lemma}\label{l.weighted}
Let $K$ be a convex set contained in the unit cube $Q=[0,1]^n$ and $\mu$ be a doubling measure which is doubling with respect to Euclidean cubes in $\R^n$ with doubling constant $\delta_\mu$. For every $\epsilon>0$ we have the estimate
\begin{align*}
\mu( \{x\in \R^n\setminus K :0\leq \dist(x,K)< \epsilon \}) \leq v_\epsilon \mu(Q),
\end{align*}
where $v_\epsilon\leq 9\delta_\mu ^{4+\lceil \log(34\sqrt n ) \rceil} \big( \log \frac{1}{\epsilon} \big)^{-1}$.
\end{lemma}
%%%%%%%%%%%%%%%%%%%%%%%%%%%%%% LEMMA LEMMA LEMMA
We immediately get the following corollary.

%%%%%%%%%%%%%%%%%%%%%%%%%%%%%% COROLLARY COROLLARY COROLLARY
\begin{corollary} Let $m$ be a positive integer and $\{Q_j\}_j$ denote the dyadic cubes of sidelength $2^{-m}$, contained in $Q$ and disjoint from $K$. Then
\begin{align*}
		\mu(\bigcup_j Q_j)+\mu(K)\geq \xi_m \mu(Q),
\end{align*}
where $\xi_m=1-v_{\sqrt n 2^{-m}}\to 1$ as $m\to +\infty$ and $v_\epsilon$ is as in Lemma~\ref{l.weighted}.
\end{corollary}
%%%%%%%%%%%%%%%%%%%%%%%%%%%%%% COROLLARY COROLLARY COROLLARY

%%%%%%%%%%%%%%%%%%%%%%%%%%%%%% PROOF PROOF PROOF
\begin{proof} Suppose that $x\in Q\setminus K$ and $\dist(x,K)\geq \sqrt{n}2^{-m}$. Then since the cubes $Q_j$ have diameter less than $\sqrt{n} 2^{-m}$ we have that $x\in Q_j$ for some $j$. Thus
	\begin{align*}
 Q\setminus  \big( \bigcup_j Q_j \cup K\big)\subset	\{x\in Q\setminus K: 0\leq  \dist(x,K)<  \sqrt{n} 2^{-m}\}.
	\end{align*}
Using Lemma~\ref{l.weighted} and the previous inclusion we conclude
\begin{align*}
\mu	\big( \bigcup_j Q_j \big)+\mu(K)\geq (1-v_{\sqrt n 2^{-m}})\mu(Q)
\end{align*}
which is the desired estimate with $\xi_m=1-v_{\sqrt n 2^{-m}}$.
\end{proof}
%%%%%%%%%%%%%%%%%%%%%%%%%%%%%% PROOF PROOF PROOF

Some remarks are in order concerning Lemma~\ref{l.weighted} and its corollary above. If $\mu$ is replaced by the Lebesgue measure then Lemma~\ref{l.weighted} appears in \cite{HS}*{Lemma 2} and is one of the main ingredients for the proof of the Lebesgue-measure analogue of Lemma~\ref{l.taubassoc}. In the Lebesgue measure case, Lemma~\ref{l.weighted} is a simple calculation that crucially depends on the fact that the Lebesgue measure of an ``annulus'' around a convex set can be calculated as the sum of the $(n-1)$-dimensional Hausdorff measures of the boundaries of an increasing sequence of convex sets. Since the $(n-1)$-dimensional Hausdorff measure of the boundary of a convex set contained in $Q$ is at most $2n$ this yields the desired estimate. Under the presence of a general doubling measure the proof of such a result is more involved. In order not to divert the attention from the proof of Lemma~\ref{l.taubassoc} we postpone the proof of Lemma~\ref{l.weighted} until \S~\ref{s.proofoflemma}.

Let $E$ be a set in $\R^n$, $\sigma\in\R^n$ and $c>0$. Remember the notations $\tau_\sigma E = \{x+\sigma:\, x\in E\}$ and $\dil_c E=\{cx: \, x\in E\}$. In the following lemma we iterate the construction of Lemma~\ref{l.weighted} in order to get ``many'' disjoint homothetic copies of a convex set $B$ inside its associated rectangle $R_B$, with diameters bounded away from zero and whose union captures a big portion of the measure of $R_B$.

Remember that $\xi_m$ is the constant appearing in Lemma~\ref{l.weighted}, $\rho= \Delta_\mu ^{- \lceil  \frac{3}{2}\log  n \rceil }<1$ is the constant from Lemma~\ref{l.convrect}, and $\delta_\mu $ is the doubling constant of $\mu$ with respect to $\mathfrak G_{\mathfrak B}$.
%%%%%%%%%%%%%%%%%%%%%%%%%%%%%% LEMMA LEMMA LEMMA
\begin{lemma}\label{l.homcopies} Let $B\in\mathfrak B$ be a convex set in let $R=R_B\in \mathfrak G_{\mathfrak B}$ be the associated rectangle of $B$ so that $B\subset R$. Let $m$ be a a large positive integer so that $\xi_m>\rho$. For every positive integer $N$ there exists a set $B_N\subset R$ with the following properties:
	\begin{itemize}
	\item[(i)] The set $B_N$ is a finite union of pairwise disjoint homothetic copies of $B$. That is, $B_N=\cup_\alpha B_\alpha $, each $B_\alpha\subset R$ is an homothetic copy of $B$ and the $B_\alpha$'s are pairwise disjoint.
	\item[(ii)] Let $ B_\alpha$ be a homothetic copy of $B$ in $B_N$  and let $R_{ B_\alpha}$ be the associated rectangle of $ B_\alpha$. Then $R_{ B_\alpha}$ is a ``dyadic rectangle'' in $\mathcal D_R$ which is at most $mN$ generations inside $R$. This means that $R_{ B_\alpha} ^{(s)}=R$ for some non-negative integer $s\leq Nm$.
	\item[(iii)] For the $\mu$-measure of $B_N$ we have the estimate
	\begin{align*}
	\mu(B_N)\geq \rho \frac{1-\Psi^{N+1}}{1-\Psi}\mu(R),
	\end{align*}
where $\Psi\coloneqq {\xi_m-\rho} $.
	\end{itemize}
\end{lemma}
%%%%%%%%%%%%%%%%%%%%%%%%%%%%%% LEMMA LEMMA LEMMA

%%%%%%%%%%%%%%%%%%%%%%%%%%%%%% PROOF PROOF PROOF
\begin{proof} By the discussion before Lemma~\ref{l.weighted} concerning affine invariance we can assume that $R=Q=[0,1]^n$ and $B=K\subset Q$. We will thus construct the set $K_N=B_N$ as in the statement of the lemma, assuming everything takes place inside the unit cube $Q$.
	
	 Let $\{Q_j\}_j$ denote the pairwise disjoint cubes from Lemma~\ref{l.weighted} which satisfy $\mu(\cup_j Q_j)+\mu(K)\geq \xi_m\mu(Q)$. It is essential to note that $\xi_m$ depends only on the dimension, the doubling constant of $\mu$ and $m$, and that $\xi_m\to 1$ as $m\to +\infty$. Remember that by Lemma~\ref{l.convrect} we have for every $B\in \mathfrak B$ with associated rectangle $R_B$ that $\mu(B)\geq \rho \mu (R_B)$. Since we have reduced everything to the case $K\subset Q$ this means that we have
	\begin{align}\label{e.doubling}
		\mu(V)\geq \rho \mu(Q_V)
	\end{align}
for all the convex sets $V\subset Q$ which are homothetic copies of $K$ and for $Q_V\supset V$ being the associated cube of $V$. Throughout this proof (and for the rest of the paper) we assume that $m$ is sufficiently large, depending upon the dimension, $\delta_\mu$ and $\Delta_\mu$ only, so that $\xi_m-\rho>0$.
	
	 Consider $\{\sigma_j\}_j\subset \R^n$ such that $Q_j\eqqcolon \tau_{\sigma_j} \dil_{2^{-m}}Q$ and set $K_{1,j}\coloneqq \tau_{\sigma_j}\dil_{2^{-m}}K$. Observe that $Q_j$ is the associated cube of the convex set $K_{1,j}$, just like $Q$ is the associated rectangle of $K$. Combining this observation with \eqref{e.doubling} we see that $\mu(K_{1,j})\geq \rho \mu(Q_j)$ and thus
\begin{align*}
	\mu(\bigcup_{j}K_{1,j}) \geq \rho \mu(\bigcup_{j}Q_j).
\end{align*}
By Lemma~\ref{l.weighted} we estimate
\begin{align*}
	\mu(\bigcup_{j}K_{1,j})+\mu(K)& \geq\rho\big(  \xi_m \mu(Q)-\mu(K)\big)+\mu(K)
	\\
	&= \rho \xi_m \mu(Q)+(1-\rho)\mu(K).
\end{align*}
Now $\rho<1$ and $\mu(K)\geq \rho \mu(Q)$ thus
\begin{align*}
		\mu(\bigcup_{j}K_{1,j})+\mu(K)& \geq \rho\xi_m  \mu(Q)+ (1-\rho)\rho \mu(Q)
		\\
		&=\rho  (1+ \xi_m-\rho)\mu(Q).
\end{align*}
We call $\Psi\coloneqq \xi_m-\rho$ since this quantity will appear quite a lot in what follows. Observe that $0<\Psi<1$. Now let $K_1\coloneqq K\cup\bigcup_\ell K_{1,\ell}$ and  $K_{2,j}\coloneqq \tau_{\sigma_j}\dil_{2^{-m}} K_1=K_{1,j}\cup \bigcup_\ell \tau_{\sigma_j}\dil_{2^{-m}} K_{1,\ell}$. The previous estimate reads
\begin{align*}
	\mu(K_1)\geq \rho(1+\Psi)\mu(Q).
\end{align*}
We iterate the estimate of Lemma~\ref{l.weighted} as follows:
\begin{align*}
		\mu(\bigcup_{j}K_{2,j})+\mu(K)& \geq \mu(\bigcup_{j}K_{1,j})+ \sum_j \mu(\bigcup_\ell \tau_{\sigma_j}\dil_{2^{-m}} K_{1,\ell})+\mu(K)
		\\
		&\geq \mu(\bigcup_{j}K_{1,j}) + \rho\sum_j \mu (\bigcup_\ell \tau_{\sigma_j}\dil_{2^{-m}} \tau_{\sigma_\ell} \dil_{2^{-m}} Q)+\mu(K)
		\\
		&=  \mu(\bigcup_{j}K_{1,j}) + \rho\sum_j \sum_\ell \mu (\tau_{\sigma_j}\dil_{2^{-m}} Q_\ell)+\mu(K)
		\\
		&\geq \mu(\bigcup_{j}K_{1,j}) +   \rho \sum_j \big(\xi_m\mu(Q_j)-\mu(K_{1,j})\big)+\mu(K)
		\\
		&= (1- \rho )\mu(\bigcup_{j}K_{1,j})+\rho\xi_m\sum_j \mu(Q_j)+\mu(K)
		\\
		&\geq (1-\rho)\mu(\bigcup_{j}K_{1,j}) + \rho\xi_m \big(\xi_m\mu(Q)-\mu(K)\big)+\mu(K).
\end{align*}
Noting that the coefficient in front of $\mu(K)$ is positive and using $\mu(K)>\rho \mu(Q)$ together with the lower bound for $\mu(\bigcup_j K_{1,j})$ we can conclude
\begin{align*}
		\mu(\bigcup_{j}K_{2,j})+\mu(K)\geq \rho\big( 1+\Psi+\Psi^2 \big)\mu(Q).
\end{align*}
Thus $K_2\coloneqq K\cup\bigcup_\ell K_{2,\ell}$ satisfies
\begin{align*}
	\mu(K_2)\geq \rho(1+\Psi+\Psi^2)\mu(Q)
\end{align*}
We continue inductively. If $K_\nu$ has been defined and satisfies $\mu(K_\nu)\geq \rho(1+\Psi+\cdots+\Psi^\nu)\mu(Q)$ we can set $K_{\nu+1,j}\coloneqq \tau_{\sigma_j}\dil_{2^{-m}}K_\nu$ and $K_{\nu+1}\coloneqq K\cup \bigcup_\ell K_{\nu,\ell} $ and in the same fashion we show that
\begin{align*}
	\mu(K_{\nu+1})\geq \rho(1+\Psi+\cdots+\Psi^\nu+\Psi^{\nu+1})\mu(Q).
\end{align*}
We conclude that for every positive integer $N$ we have $\mu(K_N)\geq \rho \frac{1-\Psi^{N+1}}{1-\Psi}\mu(Q) $.

Going back from the unit cube $Q$ to a general rectangle $R$, if $Q=T(R)$ for a bijection $T:\R^n\to \R^n$ then the desired set $B_N$ is just $B_N\coloneqq T^{-1}K_N$.
\end{proof}
%%%%%%%%%%%%%%%%%%%%%%%%%%%%%% PROOF PROOF PROOF
In the following lemma we make appropriate selections of the parameters $m,N$ as these appear in the statement of Lemma~\ref{l.homcopies}.

%%%%%%%%%%%%%%%%%%%%%%%%%%%%%% LEMMA LEMMA LEMMA
\begin{lemma}\label{l.mN} Let $\eta\in(0,1)$ and $\alpha\in(\eta,1)$ and consider the parameters $N,m$ and $\Psi$ from Lemma~\ref{l.homcopies}. There exists a choice of $m,N$ so that
	\begin{align*}
	\rho\frac{1-\Psi^{N+1}}{1-\Psi}\geq \frac{1-\alpha}{1-\eta}.
	\end{align*}
\end{lemma}
%%%%%%%%%%%%%%%%%%%%%%%%%%%%%% LEMMA LEMMA LEMMA

%%%%%%%%%%%%%%%%%%%%%%%%%%%%%% PROOF PROOF PROOF
\begin{proof} Remember that $\Psi=\xi_m-\rho$. We begin by fixing $m$ large enough so that $\xi_m>\rho$ and
	\begin{align*}
\frac{\rho}{1-\Psi}=\frac{\rho}{1-(\xi_m-\rho)}>\frac{1+\frac{1-\alpha}{1-\eta}}{2}.
	\end{align*}
For the previous inequality to be true it is enough to define $m$ so that
\begin{align*}
	\xi_m > 1-\rho \frac{1-\frac{1-\alpha}{1-\eta}}{1+\frac{1-\alpha}{1-\eta}}=1-\rho\frac{\alpha-\eta}{2-\alpha-\eta}.
\end{align*}
This is always possible since $0<\frac{\alpha-\eta}{2-\alpha-\eta}<1$ and $0<\rho<1$. With this value of $m$ fixed we now let the positive integer $N$ be large enough so that
\begin{align*}
	\frac{1+\frac{1-\alpha}{1-\eta}}{2}(1-\Psi^{N+1})=\frac{1+\frac{1-\alpha}{1-\eta}}{2}(1-(\xi_m-\rho)^{N+1})>\frac{1-\alpha}{1-\eta}.
\end{align*}
A straightforward calculation shows that it is enough to take
\begin{align*}
	N\coloneqq \bigg\lceil \frac{\log \big( \frac{2-\alpha-\eta}{\alpha-\eta}\big)}{\log\big(\frac{1}{\xi_m-\rho}\big)}\bigg\rceil.
\end{align*}
These values of $m$ and $N$ prove the statement of the lemma.
\end{proof}
%%%%%%%%%%%%%%%%%%%%%%%%%%%%%% PROOF PROOF PROOF

We are now ready to give the proof of Lemma~\ref{l.taubassoc}.

%%%%%%%%%%%%%%%%%%%%%%%%%%%%%% PROOF PROOF PROOF
\begin{proof}[Proof of Lemma~\ref{l.taubassoc}] Let us assume that $M_{\mathfrak B,\mu}$ satisfies the Tauberian condition \eqref{e.convex} for a fixed level $\gamma\in(0,1)$ and let $\beta\in(\gamma,\alpha)$. As in the proof of Theorem~\ref{t.doubling} any such $\beta$ will work equally well, but a concrete choice is $\beta\coloneqq \sfrac{(\gamma+\alpha)}{2}$. Then the Tauberian condition \eqref{e.convex} implies
	\begin{align}\label{e.taubbeta}
 \nu(\{x\in\R^n:M_{\mathfrak B,\mu}(\ind_E)(x)\geq \beta \})\leq c_{\mathfrak G_{\mathfrak B},\gamma ,\nu} ^\mu \nu(E)
	\end{align}
for any $\mu$-measurable set $E\subset \R^n$. For every $y\in\{x\in \R^n:M_{\mathfrak G_{\mathfrak B},\mu}(\ind_E)(x)>\alpha\}$ there exists a rectangle $R=R_y\in\mathfrak G_{\mathfrak B}$ with $R\ni y$ and
	\begin{align}\label{e.rec}
		\frac{1}{\mu(R)}\int_R \ind_E(y) d\mu(y)>\alpha.
	\end{align}
We will use the basic inclusion
	\begin{align*}
		H_\alpha \coloneqq \{x\in \R^n:M_{\mathfrak G_{\mathfrak B},\mu}(\ind_E)(x)>\alpha\} \subset \bigcup_{y\in H_\alpha} R_y.
	\end{align*}
Denoting
\begin{align*}
H_{\mathfrak B,\beta} ^0 (E)\coloneqq E\quad\text{and}\quad	\mathcal H^k _{\mathcal B,\beta}(E)\coloneqq \{x\in\R^n: M_{\mathfrak B,\mu}(\ind_{\mathcal H ^{k-1} _{\mathfrak B,\beta}(E)})(x)\geq \beta\},\quad k\geq 1,
\end{align*}
we will show that $\cup_{y} R_y \subset \mathcal H_{\mathfrak B,\beta} ^k(E)$ for some positive integer $k$.

To this end let $R=R_y$ be one of these rectangles. Observe that there is some $B\in\mathfrak B$ such that $R=R_B$ is the associated rectangle of the convex set $B$ and $R\supset B$. We now consider the set $B_N$ as in Lemma~\ref{l.homcopies} with the choice of parameters $m,N$ provided by Lemma~\ref{l.mN} applied for $\eta=\beta$. Observe that we have $\mu(B_N)\geq \frac{1-\alpha}{1-\beta}\mu(R_B)$ and $B_N$ is a disjoint union of homothetic copies of $B$ inside $R_B$, with each homothetic copy having associated a rectangle which is dyadic, and at most $Nm$ generations ``inside'' $R_B$. We claim that for at least one of the homothetic copies of $B$ forming $B_N$, say $\tilde B$, we have
\begin{align*}
	\frac{1}{\mu(\tilde B)}\int_{\tilde B}\ind_E(y)d\mu(y)\geq \beta.
\end{align*}
Indeed, if this is not the case then we would have
\begin{align*}
	\mu(E\cap R)&\leq \mu(E\cap B_N)+\mu(E\cap R \setminus B_N)< \beta \mu(B_N)+\mu(R)-\mu(B_N)
	\\
	&=\mu(R)-(1-\beta)\mu(B_N)\leq \mu(R)-(1-\beta)\frac{1-\alpha}{1-\beta}\mu(R)=\alpha\mu(R)
\end{align*}
which contradicts~\eqref{e.rec}. The previous claim just proved immediately implies that
\begin{align}\label{e.1step}
 \tilde B \subset \mathcal H_{\mathcal B,\beta} ^1 (E).
\end{align}

If $R_{\tilde B}$ is the associated rectangle of $\tilde B$ we get by the construction of Lemma~\ref{l.homcopies} that $R_{\tilde B}$ is a dyadic rectangle which is at most $Nm$ generations ``inside'' $R$. Thus $2^{Nm+1}R_{\tilde B}\supset R $. Remembering that $R_{\tilde B} \subset  n^\frac{3}{2}  \tilde B $ we arrive at
\begin{align*}
	 \tilde B  \subset R \subset 2^{Nm+1+\frac{3}{2}\lceil \log n \rceil}   \tilde B .
\end{align*}
From the doubling property of $\mu$ with respect to $ \mathfrak B $ we now get that
\begin{align*}
	\mu( 2^{Nm+1+\frac{3}{2}\lceil \log n \rceil}   \tilde B )\leq \Delta_\mu ^{Nm+1+\lceil \frac{3}{2}\log n \rceil} \mu(  \tilde B ).
\end{align*}

We define a nested sequence of homothetic copies of $\tilde B$ as follows. Let $\tilde B_0\coloneqq \tilde B$. Assuming we have defined $\tilde B_0,\ldots,\tilde B_j$ then we set $\tilde B_{j+1}\coloneqq c_j \tilde B_j$ where $c_j>1$ is chosen so that
\begin{align}\label{e.measureratio}
\frac{\mu(\tilde B_{j+1})}{\mu(\tilde B_j)}=\frac{1}{\beta}>1.
\end{align}
This is possible because of the continuity of the measure $\mu$ proved in (iv) of Proposition~\ref{p.noboundarymass}.
Here remember that $c\tilde B$ denotes dilation with respect to the center of the John ellipsoid of $\tilde B$. It is not hard to see that $ \tilde B\subset c\tilde B$ whenever $c>1$. We define $k=k_{\alpha,\beta,n,\mu}$ to be the smallest positive integer such that
\begin{align*}
 \big(\frac{1}{\beta}\big) ^{k-1} \geq \Delta_\mu ^{Nm+1+ \lceil \frac{3}{2}\log  n \rceil  }.
\end{align*}
Observe that for this it suffices to set
\begin{align*}
 k\coloneqq 1+\bigg\lceil \frac{\log \Delta_\mu}{\log(1/\beta)}\bigg\rceil \big\lceil Nm+1+\frac{3}{2}\log n \big\rceil,
\end{align*}
with the choices of $m,N$ given by Lemma~\ref{l.mN} with $\eta=\beta$. Then we claim that
\begin{align*}
	\tilde B_{k-1}\supseteq 2^{Nm+1+\lceil\frac{3}{2}\log n \rceil} \tilde B.
	\end{align*}
If this is not the case then necessarily $\tilde B_{k-1}\subsetneq  2^{Nm+1+\lceil\frac{3}{2}\log n \rceil} \tilde B$ since both sets in the previous inclusion are of the form $c\tilde B$ and, as already observed, one must contain the other. This would imply
\begin{align*}
\Delta_\mu^{{Nm+1+\lceil\frac{3}{2}\log n \rceil}} \geq  \frac{\mu(2^{Nm+1+\lceil\frac{3}{2}\log n \rceil} \tilde B )}{\mu(\tilde B)} >	\frac{\mu(\tilde B_{k-1})}{\mu(\tilde B)}=\big(\frac{1}{\beta}\big)^{k-1}
\end{align*}
which contradicts the choice of $k$. Now for every $j=1,\ldots,k-1$ the choice in \eqref{e.measureratio} implies that
\begin{align*}
 \tilde B_j \subset \mathcal H  ^1 _{\mathfrak B,\beta}(\tilde B_{j-1}).
\end{align*}
Iterating the previous inclusion we get
\begin{align*}
 R\subset2^{Nm+1+\lceil\frac{3}{2}\log n \rceil} \tilde B\subset \tilde B_{k-1} \subset  \mathcal H_{\mathfrak B,\beta} ^{k-1}( \tilde B)\subset\mathcal H_{\mathfrak B,\beta} ^{k}( E)
\end{align*}
by \eqref{e.1step}.

Remembering that $H_\alpha=\{x\in \R^n:M_{\mathfrak G_{\mathfrak B},\mu}(\ind_E)(x)>\alpha\}\subset \bigcup_{y\in H_\alpha} R_y$, the previous inclusion and \eqref{e.taubbeta} imply
\begin{align*}
	\nu(\{x\in \R^n:M_{\mathfrak G_{\mathfrak B},\mu}(\ind_E)(x)>\alpha\})\leq \nu(\bigcup_y R_y)\leq \nu(\mathcal H_{\mathfrak B,\beta} ^{k}( E))\leq  [c_{\mathfrak B,\gamma,\nu} ^\mu ]^k \nu(E)
\end{align*}
which is the Tauberian condition for $M_{\mathfrak G_{\mathfrak B},\mu}$ with constant $c_{\mathfrak G_{\mathfrak B},\alpha ,\nu} ^\mu\leq [c_{\mathfrak B,\gamma,\nu} ^\mu ]^k$.
\end{proof}
%%%%%%%%%%%%%%%%%%%%%%%%%%%%%% PROOF PROOF PROOF

%%%%%%%%%%%%%%%%%%%%%%%%%%%%%% SECTION SECTION SECTION
\subsection{The proof of Theorem~\ref{t.convdoubling}} It is now routine to complete the proof of Theorem~\ref{t.convdoubling}. Indeed, suppose that $M_{\mathfrak B,\mu}$ satisfies \eqref{e.convex} for some fixed level $\gamma\in(0,1)$. Then Lemma~\ref{l.taubassoc} implies that the associated operator $M_{\mathfrak G_{\mathfrak B},\mu}$ satisfies \eqref{e.taubassoc}. Since $\mathfrak G_{\mathfrak B}$ is an homothecy invariant basis of rectangles and $\mu$ is doubling with respect to that basis, Theorem~\ref{t.doubling} implies that $M_{\mathfrak G_{\mathfrak B},\mu}$ is bounded on $L^p(\nu)$ for $p>p_0$ with $p_0$ depending on $\mu, n, \gamma$ and the constant $c_{\mathfrak B,\gamma,\nu} ^\mu,$ in the Tauberian condition \eqref{e.convex}. By Lemma~\ref{l.convrect} the operators $M_{\mathfrak G_{\mathfrak B},\mu}$ and $M_{\mathfrak B,\mu}$ are pointwise comparable. We conclude that $M_{\mathfrak B,\mu}$ is bounded on $L^p(\nu)$ for $p>p_0$.

%%%%%%%%%%%%%%%%%%%%%%%%%%%%%% SECTION  SECTION SECTION

\subsection{The proof of Lemma~\ref{l.weighted}\label{s.proofoflemma}} This section is dedicated to the proof of Lemma~\ref{l.weighted}. Remember that $Q=[0,1]^n$ is the unit cube in $\R^n$ and $K\subset Q$ is a convex set and our purpose is to estimate the $\mu$-measure of the ``annulus'' $\{x\in \R^n\setminus K : 0\leq \dist(x,K) < \epsilon\}$. The estimate of the lemma is only interesting when $\epsilon$ is small so it is without loss of generality to assume that $0<\epsilon<\frac{1}{2^8}$. We consider the larger cube $L\coloneqq 16 Q$ which is the cube with the same center as $Q$ and and sidelength equal to $16$. Now let us fix a positive integer $k\geq 8$ such that $2^{-k-1}\leq \epsilon < 2^{-k}$. It  obviously suffices to estimate the measure of the ``annulus''
\begin{align*}
 A_k\coloneqq \{x\in \R^n\setminus K  : 0\leq \dist(x,K) < 2^{-k}\}.
\end{align*}
Let $x\in A_k$ and consider $p=p_x\in \partial K$ such that $|p-x|\leq 2^{-k}$. Let $H_p$ be a supporting hyperplane through $p$ and $H_p ^-$ be the half space defined by $H_p$ and such that $H_p ^- \cap K=\emptyset$. We now consider the half line $\ell_p$ which is emanating from the point $p$, is perpendicular to $H_p$ and is contained in $H_p ^-$, and suppose that $\ell_p$ meets $\partial L$ at some point $b_p$. Since $K\cap H_p ^-=\emptyset$ and $\ell_p$ is perpendicular to $H_p$, we have for $z\in \ell_p$:
\begin{align}\label{e.distances}
\dist(z, K)= \dist(z,H_p)=|z-p|,\quad z\in\ell_p.
\end{align}
%%%%%%%%%%%%%%%%%%%%%%%%%%%%%% FIGURE FIGURE FIGURE
\begin{figure}[htb]
\centering
 \def\svgwidth{300pt}
\begingroup%
  \makeatletter%
  \providecommand\color[2][]{%
    \errmessage{(Inkscape) Color is used for the text in Inkscape, but the package 'color.sty' is not loaded}%
    \renewcommand\color[2][]{}%
  }%
  \providecommand\transparent[1]{%
    \errmessage{(Inkscape) Transparency is used (non-zero) for the text in Inkscape, but the package 'transparent.sty' is not loaded}%
    \renewcommand\transparent[1]{}%
  }%
  \providecommand\rotatebox[2]{#2}%
  \ifx\svgwidth\undefined%
    \setlength{\unitlength}{517.46641938bp}%
    \ifx\svgscale\undefined%
      \relax%
    \else%
      \setlength{\unitlength}{\unitlength * \real{\svgscale}}%
    \fi%
  \else%
    \setlength{\unitlength}{\svgwidth}%
  \fi%
  \global\let\svgwidth\undefined%
  \global\let\svgscale\undefined%
  \makeatother%
  \begin{picture}(1,0.88310078)%
    \put(0,0){\includegraphics[width=\unitlength]{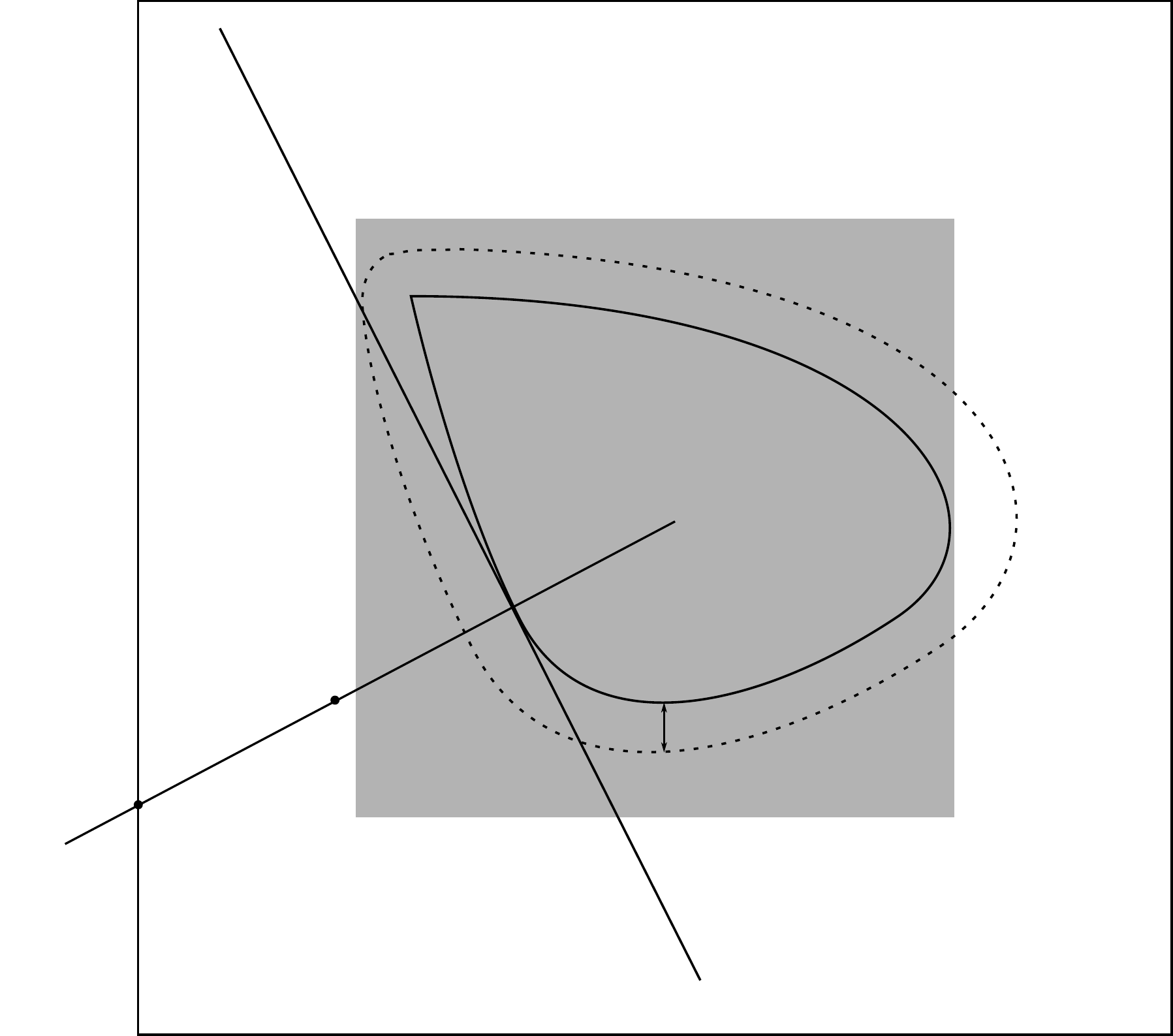}}%
    \put(0.02307484,0.19929764){\color[rgb]{0,0,0}\makebox(0,0)[lt]{\begin{minipage}{0.11369102\unitlength}\raggedright \end{minipage}}}%
    \put(0.12714143,0.15862914){\color[rgb]{0,0,0}\makebox(0,0)[lb]{\smash{$b_p$}}}%
    \put(0.45180018,0.35342439){\color[rgb]{0,0,0}\makebox(0,0)[lb]{\smash{$p$}}}%
    \put(0.74863103,0.195733){\color[rgb]{0,0,0}\makebox(0,0)[lb]{\smash{$Q$}}}%
    \put(0.89395447,0.03494962){\color[rgb]{0,0,0}\makebox(0,0)[lb]{\smash{$L$}}}%
    \put(-0.00272207,0.15244517){\color[rgb]{0,0,0}\makebox(0,0)[lb]{\smash{$\ell_p$}}}%
    \put(0.236907,0.29003863){\color[rgb]{0,0,0}\makebox(0,0)[lb]{\smash{$z_p$}}}%
    \put(0.72080314,0.39516623){\color[rgb]{0,0,0}\makebox(0,0)[lb]{\smash{$K$}}}%
    \put(0.56929572,0.24984279){\color[rgb]{0,0,0}\makebox(0,0)[lb]{\smash{$2^{-k}$}}}%
    \put(0.58166368,0.0133057){\color[rgb]{0,0,0}\makebox(0,0)[lb]{\smash{$H_p$}}}%
    \put(0.46880611,0.09988137){\color[rgb]{0,0,0}\makebox(0,0)[lb]{\smash{$H_p ^-$}}}%
    \put(0.59866961,0.09833537){\color[rgb]{0,0,0}\makebox(0,0)[lb]{\smash{$H_p ^+$}}}%
  \end{picture}%
\endgroup%

\caption[A figure for the proof of Lemma~\ref{l.weighted}]{A figure for the proof of Lemma~\ref{l.weighted}}\label{f.proofofLemma}
\end{figure}
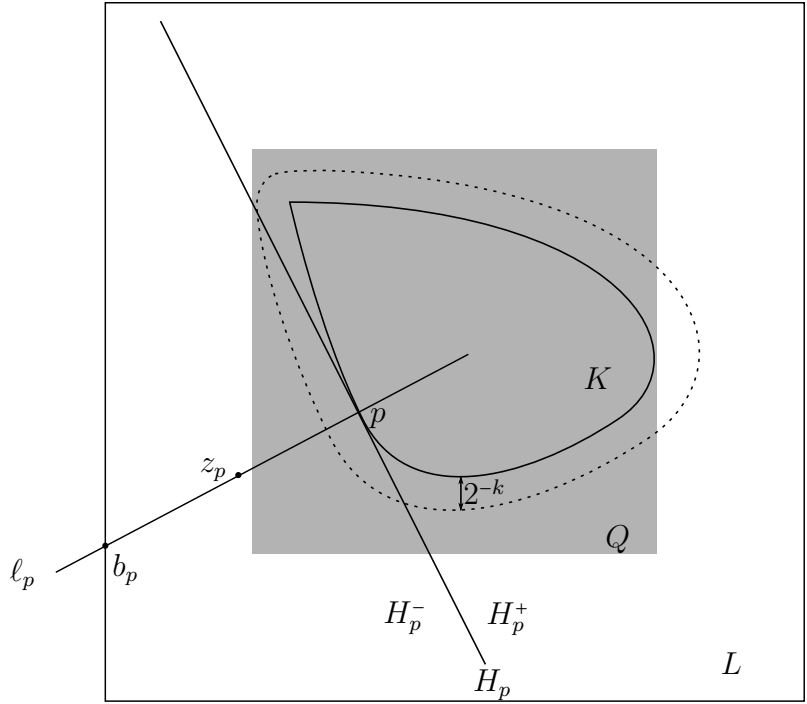
%%%%%%%%%%%%%%%%%%%%%%%%%%%%%% FIGURE FIGURE FIGURE
Thus we can conclude
\begin{align*}
 \{\dist(z, K): z\in \ell_p\}\supset [0,|b_p-p|]\supset [0, 1]
\end{align*}
since $|b_p-p|\geq   8-1/2 > 1$. Now let $2^{-j} \in [2^{-k},1]$ for some $0 \leq j\leq k$. By the continuity of the distance function there exists $z_p \in \ell_p$ such that $\dist(z_p,  K)=2^{-j}$. Let $\{S_j\}_j$ denote the Whitney cubes associated to $K$ and remember that these cubes satisfy
\begin{align*}
\bar{K}^\mathsf{c} &= \cup_j S_j,
\\
\mathring {S_j} \cap \mathring{S_{j'}} &=\emptyset\quad \text{if}\quad j\neq j',
\\
\diam  S_j  \leq \dist&(S_k , K)\leq 4\diam S_j.
\end{align*}
Here $\mathring{A}$ denotes the interior of a set $A$. See for example \cite{Stein}*{\S VI.1} for the details of the construction of the Whitney cubes. Since the Whitney cubes cover $\bar{K} ^\mathsf{c}$ and $\dist(z_p, K)>0$, there is a Whitney cube $S_p=S_{j_p}$ such that $z_p\in S_p$. We then have that
\begin{align*}
 \diam S_p \leq \dist (S_p, K)\leq \dist(z_p,  K)
\end{align*}
and on the other hand
\begin{align*}
\diam S_p \geq \frac{1}{4} \dist(S_p,  K) \geq\frac{1}{4}(\dist(z_p,K)-\diam S_p)
\end{align*}
so that $\diam S_p\geq \frac{1}{5} \dist(z_p,K) >\frac{1}{8}\dist(z_p,K)$. Thus the Whitney cube $S_p$ satisfies
\begin{align*}
  \diam S_p \leq \dist(z_p, K)=2^{-j}\leq 8 \diam S_p.
\end{align*}
By \eqref{e.distances} we also get that $\diam S_p > \frac{1}{8}\dist(z_p, K) = \frac{1}{8}| z_p-p|$. Remember that $x\in A_k$ satisfies $|x-p|\leq 2^{-k}\leq 2^{-j}$. An easy calculation now verifies that $x\in 34 \sqrt{n}S_p$, where we remember that the dilation is taken with respect to the center of $S_p$.

We have actually shown that for every $x\in A_k$ and every $2^{-j}\in [2^{-k},1]$, there exists a Whitney cube $S_x$ such that
\begin{align*}
 x\in  34\sqrt n   S_x \quad \text{and}\quad  \frac{1}{8} 2^{-j} <\diam S_x\leq 2^{-j}
\end{align*}
Let $\mathcal C_j$ denote the Whitney cubes such that $\frac{1}{8} 2^{-j} < \diam S \leq 2^{-j}$. Observe that for different $j$'s the collections $\mathcal C_{4j}$ are disjoint. We can write for every positive integer 	$j\in [0,k/4)$
\begin{align*}
 A_k \subset \bigcup_{S\in \mathcal C_{4j}}  34 \sqrt n  S,
\end{align*}
and thus
\begin{align*}
 \mu(A_k) \leq \delta_\mu ^{\lceil \log  34\sqrt n  \rceil} \mu(\bigcup_{S\in \mathcal C_{4j}}  S).
\end{align*}
Summing in $j\in[0,k/4)$ yields
\begin{align*}
\lfloor k/4 \rfloor \mu(A_k)\leq  \delta_\mu ^{\lceil \log  34\sqrt n  \rceil} \sum_{j=0} ^{\lfloor k/4 \rfloor} \mu(\bigcup_{S\in \mathcal C_{4j}}  S),
\end{align*}
with $\lfloor x \rfloor $ denoting the largest integer less or equal to $x$. Now all the Whitney cubes $S$ that appear on the right hand side of the last display satisfy $\diam(S)\leq 1$ and $\dist(S,K)\leq 4$, and thus they are all contained in $11Q\subset L$. Since the families $\mathcal C_{4j}$ are pairwise disjoint, and every family $\mathcal C_{4j}$ consists of pairwise $\mu$-a.e. disjoint cubes, the previous estimate implies
\begin{align*}
\mu(A_k)\leq \frac{8}{k} \delta_\mu ^{\lceil \log  (34\sqrt n)   \rceil} \mu(16Q)\leq \frac{8}{k}\delta_\mu ^{4+\lceil \log (34\sqrt n   ) \rceil}  \mu(Q).
\end{align*}
Since $\frac{1}{k}\leq \frac{9}{8}\frac{1}{\log\frac{1}{\epsilon}}$ this completes the proof of Lemma~\ref{l.weighted}.

%%%%%%%%%%%%%%%%%%%%%%%%%%%%%% SECTION  SECTION SECTION

\end{document}